\documentclass[12pt]{amsart}
\usepackage{amssymb}
\usepackage[shortlabels]{enumitem}
\usepackage[all]{xy}
\usepackage{xcolor}
\usepackage{graphicx}
\usepackage[margin=1in]{geometry} 
\usepackage{eucal}
\usepackage{hyperref}
\usepackage{listings}
\hypersetup{bookmarksdepth=2}
\hypersetup{colorlinks=true}
\hypersetup{linkcolor=blue}
\hypersetup{citecolor=blue}
\hypersetup{urlcolor=blue}

\numberwithin{equation}{section}
\newtheorem{theorem}[equation]{Theorem}

\newtheorem{proposition}[equation]{Proposition}
\newtheorem{lemma}[equation]{Lemma}
\newtheorem{corollary}[equation]{Corollary}
\newtheorem{problem}[equation]{Problem}

\theoremstyle{definition}
\newtheorem{remark}[equation]{Remark}

\newcommand{\arxiv}[1]{\href{http://arxiv.org/abs/#1}{{\tiny\tt arXiv:#1}}}
\newcommand{\DOI}[1]{\href{http://doi.org/#1}{\color{purple}{\tiny\tt DOI:#1}}}
\newcommand{\defn}[1]{\emph{#1}}

\newcommand{\cB}{\mathcal{B}}

\newcommand{\bC}{\mathbf{C}}

\newcommand{\cE}{\mathcal{E}}

\newcommand{\bF}{\mathbf{F}}

\newcommand{\fF}{\mathfrak{F}}

\newcommand{\cG}{\mathcal{G}}
\newcommand{\fG}{\mathfrak{G}}

\newcommand{\rK}{\mathrm{K}}

\newcommand{\fM}{\mathfrak{M}}
\newcommand{\rM}{\mathrm{M}}

\newcommand{\cP}{\mathcal{P}}
\newcommand{\fP}{\mathfrak{P}}

\newcommand{\fS}{\mathfrak{S}}

\newcommand{\bZ}{\mathbf{Z}}

\newcommand{\fa}{\mathfrak{a}}

\newcommand{\bm}{\mathbf{m}}

\newcommand{\bn}{\mathbf{n}}

\let\ol\overline
\renewcommand{\phi}{\varphi}
\DeclareMathOperator{\tr}{tr}
\DeclareMathOperator{\End}{End}
\DeclareMathOperator{\Span}{span}

\DeclareMathOperator{\Hom}{Hom}
\DeclareMathOperator{\Isom}{Isom}
\DeclareMathOperator{\Ob}{Ob}
\DeclareMathOperator{\im}{im}
\DeclareMathOperator{\dom}{dom}
\DeclareMathOperator{\diag}{diag}
\DeclareMathOperator{\srk}{\mathrm{srk}}
\DeclareMathOperator{\rank}{\mathrm{rk}}
\DeclareMathOperator{\Ind}{Ind}
\DeclareMathOperator{\Res}{Res}
\DeclareMathOperator{\Rep}{Rep}
\newcommand{\id}{\mathrm{id}}
\newcommand{\HS}{\mathrm{HS}}
\newcommand{\VS}{\mathrm{VS}}
\newcommand{\Mod}{\mathrm{Mod}}

\newcommand{\GL}{\mathbf{GL}}
\newcommand{\bone}{\mathbf{1}}
\newcommand{\transpose}[1]{\, {}^{t}#1}
\newcommand{\qbinom}[2]{\genfrac{[}{]}{0pt}{}{#1}{#2}_q}

\title{Representations of finite matrix monoids}

\author{Nate Harman}
\address{Department of Mathematics, University of Georgia, Athens, GA, USA}
\email{\href{mailto:nharman@uga.edu}{nharman@uga.edu}}
\urladdr{\url{https://www.nateharman.com/}}
\thanks{NH was supported by NSF grant DMS-2401515}

\author{Andrew Snowden}
\thanks{AS was supported by NSF grant DMS-2301871.}
\address{Department of Mathematics, University of Michigan, Ann Arbor, MI, USA}
\email{\href{mailto:asnowden@umich.edu}{asnowden@umich.edu}}
\urladdr{\url{http://www-personal.umich.edu/~asnowden/}}

\author{Elad Zelingher}
\address{Department of Mathematics, University of Michigan, Ann Arbor, MI, USA}
\email{\href{mailto:eladz@umich.edu}{eladz@umich.edu}}
\urladdr{\url{https://elad.zelingher.com/}}

\date{December 1, 2025}
\begin{document}

\begin{abstract}
Let $\fM_n$ be the multiplicative monoid of $n \times n$ matrices over a finite field. The monoid algebra $\bC[\fM_n]$ has been studied for several decades. One of the important early results is Kov\'acs' theorem that the two-sided ideal spanned by matrices of rank at most $r$ has a unit. Our most significant result is an explicit formula for this unit. Prior to our work, such a formula was only known in a few examples. We also study the module theory of $\bC[\fM_n]$. We explicitly describe the simple modules, and establish induction and restriction rules. We show that the simple decomposition of an arbitrary module can be determined using character theory of finite general linear groups; this relies on a Pieri rule of Gurevich--Howe. We also establish a version of Schur--Weyl duality for $\bC[\fM_n]$. Many of these results hold over more general coefficient fields.
\end{abstract}

\maketitle
\tableofcontents

\section{Introduction} \label{s:intro}

Let $\bF$ be a finite field with $q$ elements, let $\fM_n$ be the monoid of all $n \times n$ matrices with entries in $\bF$ under multiplication, and let $k[\fM_n]$ be the monoid algebra of $\fM_n$ over a field $k$; we assume throughout that $q$ is non-zero in $k$. The algebra $k[\fM_n]$ has been an important object of study within the representation theory of monoids for the last several decades; see, for example, \cite{Faddeev, Kovacs, Kuhn2, OP, Steinberg}. In this paper, we establish some new results about this algebra and its module theory, motivated by applications to tensor categories (see \S \ref{ss:motivation}).

\subsection{The unit formula}

Kov\'acs \cite{Kovacs} proved an important structure theorem (previously announced by Faddeev \cite{Faddeev}): letting $\fG_n=\GL_n(\bF)$, there is an algebra isomorphism
\begin{displaymath}
k[\fM_n] \cong \prod_{r=0}^n \rM_{N(r)}(k[\fG_r]),
\end{displaymath}
where $N(r)$ is the number of $r$-dimensional subspaces of $\bF^n$. As a corollary, if $\# \fG_n$ is non-zero in $k$ then $k[\fM_n]$ is semi-simple; this was independently proved by Okni\'nski and Putcha \cite{OP} for $k=\bC$. The subspace $\fa_{n,r}$ of $k[\fM_n]$ spanned by matrices of rank at most $r$ is a two-sided ideal, and the key to Kov\'acs' proof is showing that $\fa_{n,r}$ has a unit.

Our most significant result is an explicit formula for this unit. To state it, we must introduce a piece of terminology: a matrix $m$ is \defn{semi-idempotent} if its eigenvalues are~0 and~1, and its 1-eigenspace coincides with the generalized 1-eigenspace. In other words, in an appropriate basis, $m$ is block diagonal, where one block is the identity and the other block is nilpotent. We define the \defn{stable rank} of a matrix $m$, denoted $\srk{m}$, to be the rank of $m^i$ for $i \gg 0$. For a semi-idempotent matrix, the stable rank is the size of the identity block. We can now state our unit formula (see also Theorem~\ref{thm:unit}):

\begin{theorem}
The unit of $\fa_{n,r}$ is the element
\begin{displaymath}
q^{-\left(n-1\right)r} \mu(r) \sum_m \mu(\srk{m}) \qbinom{n-1-\rank m}{r - \rank m} [m],
\end{displaymath}
where the sum is over semi-idempotent matrices $m$, $\mu(i) = (-1)^i q^{\binom{i}{2}}$, and the first bracketed expression is the $q$-binomial coefficient.
\end{theorem}

Kov\'acs' original proof shows that the unit is a linear combination of semi-idempotent matrices. Prior to our work, it seems that very little was known about the coefficients in this expression. Kuhn \cite[\S 3.1]{Kuhn2} made some observations, which we make use of, and determined the unit explicitly for $n=2$ and $q=2$. Steinberg \cite{Steinberg} gave an exposition of Kov\'acs' proof and included the case Kuhn analyzed as an exercise. As far as we know, this is the only case where an explicit formula for the unit has appeared in the literature.

\begin{remark}
The quantity $\mu(r)$ appearing in the above formula is closely related to the M\"obius function on the lattice of subspaces of $\bF^n$. This interpretation does not figure into our proof, but it seems likely to us that it is somehow relevant. It would be interesting to figure out how.
\end{remark}

\subsection{Module theory}

In addition to the unit formula, we prove some results about the module theory of $k[\fM_n]$. For this, we suppose that $k$ has characteristic~0, so that $k[\fM_n]$ is semi-simple. It follows from Kov\'acs' theorem that simple $k[\fM_n]$-modules correspond to simple $k[\fG_r]$-modules for $0 \le r \le n$. For a simple $k[\fG_r]$-module $\pi$, we let $L_n(\pi)$ denote the corresponding simple $k[\fM_n]$-module.

We give an explicit description of $L_n(\pi)$ (\S \ref{ss:simple}), and describe how these modules behave under certain induction and restriction operations (\S \ref{ss:res-group}). We also show how one can determine the multiplicity of $L_n(\pi)$ in an arbitrary $k[\fM_n]$-module using knowledge of characters of $\fG_n$ (Theorem~\ref{thm:multiplicity}). This relies on a version of the Pieri rule for $\fG_n$ due to Gurevich and Howe \cite{GH1}. Finally, we prove a version of Schur--Weyl duality in this context. Let $V=k[\bF^n]$ be the basic permutation representation of $\fM_n$. We determine the full centralizer of $k[\fM_n]$ in $\End(V^{\otimes m})$ (Theorem~\ref{thm:doublecentralizer}), and determine the simple decomposition of $V^{\otimes m}$ (Proposition~\ref{prop:SWdecomposition}).

\subsection{Motivation} \label{ss:motivation}

In forthcoming work \cite{HSZ}, we study some interesting semi-simple pre-Tannakian tensor categories called \defn{Jacobi categories}. The endomorphism algebras of basic objects in these categories are (essentially) the algebras $k[\fM_n]$. Thus one needs to understand these algebras well in order to effectively analyze the Jacobi categories. In particular, we use results from this paper to determine the categorical dimension of simple objects in Jacobi categories. This was our primary source of motivation.

Another source of motivation is simply that the algebras $k[\fM_n]$ are quite natural, and therefore intrinsically interesting. This work can be seen as a contribution to the extensive literature on them, and related algebras. We mention a few specific pieces of work; this discussion is by no means complete.
\begin{itemize}
\item Obviously, our work is closely related to the work of Faddeev \cite{Faddeev}, Kov\'acs \cite{Kovacs}, and Okni\'nski--Putcha \cite{OP}. Kuhn \cite{Kuhn2} also built on Kov\'acs work and proved a useful result about the units of rank ideals, which we use (Lemma~\ref{lem:kuhn-criterion-for-unit}). Steinberg \cite{Steinberg2} has also recently given a new proof of Kov\'acs' theorem.
\item The \defn{rook monoid} is closely related to the monoid $\fM_n$; in fact, the rook monoid is to $\fM_n$ as the symmetric group $\fS_n$ is to the general linear group $\fG_n$. The representation theory of the rook monoid was studied by Munn \cite{Munn1, Munn2} and Solomon \cite{Solomon3}. Some of our results mirror results from those papers.
\item Solomon \cite{Solomon1} also introduced a \defn{$q$-rook algebra}; this is analogous to the Iwahori--Hecke algebra deforming $k[\fS_n]$. It has been studied in a number of papers, such as \cite{Halverson, HR, JWL, Solomon4}.
\item There is a well-developed theory of reductive monoids, generalizing the theory of reductive groups. The simplest example is $\fM_n$. See \cite{Solomon2} for a survey.
\end{itemize}
A textbook treatment of some of these topics is given in Steinberg's book \cite{Steinberg}.

An outline of the remainder of the paper is as follows. Section \ref{s:preliminaries} reviews some background material on $k[\fM_n]$ and the representation theory of finite general linear groups $\fG_n = \GL_n(\bF)$.  Section \ref{s:semi} discusses semi-idempotent matrices, and proves some algebraic and enumerative results about them.  Section \ref{s:unit} proves the main theorem, giving the formula for the unit element $\mathfrak{a}_{n,r}$. Section \ref{s:module-theory} describes the representation theory of $k[\fM_n]$ in detail. Finally, Section \ref{s:schur-weyl} establishes a version Schur-Weyl duality for $k[\fM_n]$. 

\subsection{Notation}

We list the most important notation.
\begin{description}[align=right,labelwidth=2.5cm,leftmargin=!]
\item[ $k$ ] The coefficient field
\item[ $\bF$ ] A finite field of cardinality $q$
\item[ $\fM_n$ ] The monoid of $n \times n$ matrices over $\bF$ under multiplication
\item[ $\fa_{n,r}$ ] The ideal of $k[\fM_n]$ of rank $\le r$ matrices
\item[ $e_{n,r}$ ] The standard rank $r$ idempotent matrix in $\fM_n$
\item[ $\fG_n$ ] The group $\GL_n(\bF)$ of invertible $n \times n$ matrices
\item[ $\mu(n)$ ] The quantity $(-1)^n q^{\binom{n}{2}}$
\item[ $\srk(m)$ ] The stable rank of the matrix $m$, i.e., the rank of $m^n$ for $n \gg 0$.
\end{description}

\section{Preliminaries}\label{s:preliminaries}

\subsection{Groupoid algebras}

To describe the structure of $k[\fM_n]$, it will be convenient to use the language of groupoids. Recall that a \defn{groupoid} is a category in which all morphisms are isomorphisms. We say that a groupoid is \defn{connected} if any two objects are isomorphic, and \defn{finite} if there are finitely many objects and morphisms.

Let $\cG$ be a finite groupoid. For two objects $x$ and $y$, we write $\Isom(x,y)$ for the finite set of (iso)morphisms from $x$ to $y$. We define the \defn{groupoid algebra} $k[\cG]$ as follows. As a vector space, we have
\begin{displaymath}
k[\cG] = \bigoplus_{x,y \in \Ob(\cG)} k[\Isom(x,y)].
\end{displaymath}
For an isomorphism $a \colon x \to y$, we write $[a]$ for the corresponding basis vector of $k[\cG]$. We define the multiplication law on $k[\cG]$ by declaring $[b][a]$ to be $[b \circ a]$ if the composition $b \circ a$ is defined, and~0 otherwise. One readily verifies that $k[\cG]$ is an associative and unital $k$-algebra. The unit element is given by $\sum_{x \in \Ob(\cG)} [\id_x]$.

We require the following simple result that expresses a groupoid algebra in terms of an ordinary group algebra. This result is well-known (see, e.g., \cite[Theorem~8.15]{Steinberg}), but we include a proof for the sake of completeness.

\begin{proposition} \label{prop:groupoid}
Suppose $\cG$ is connected and has $n$ objects, and let $G$ be the automorphism group of some object. Then $k[\cG]$ is isomorphic to $\rM_n(k[G])$.
\end{proposition}

\begin{proof}
Let $x_1, \ldots, x_n$ be the objects, and suppose $G$ is the automorphism group of $x_1$. For each $1 \le i \le n$, choose an isomorphism $f_i \colon x_1 \to x_i$, with $f_1$ the identity. Let $e_{i,j} \in M_n(k)$ be the usual elementary matrix. We define a map
\begin{displaymath}
\phi \colon k[\cG] \to \rM_n(k[G])
\end{displaymath}
as follows: for $a \in \Isom(x_i, x_j)$, we put $\phi([a])=[f_j^{-1}af_i] e_{j,i}$. Here $[f_j^{-1}af_i]$ is a basis element of $k[G]$. One easily verifies that this is an algebra isomorphism.
\end{proof}

\subsection{Kov\'acs theorem} \label{ss:kovacs}

As we saw in \S \ref{s:intro}(a), Kov\'acs proved $k[\fM_n]$ is a product of matrix algebras over certain group rings. We now reformulate this result in a more canonical manner.

Let $\cG_{n,r}$ be the groupoid whose objects are the $r$-dimensional subspaces of $\bF^n$, and whose morphisms are isomorphisms of vector spaces. Define a linear map
\begin{displaymath}
\psi_r \colon k[\fM_n] \to k[\cG_{n,r}], \qquad
\psi_r([m]) = \sum_U [m\restriction_U]
\end{displaymath}
where the sum is over $r$-dimensional subspaces $U$ of $\bF^n$ such that $m(U)$ also has dimension $r$. In this formula, $m \restriction_U$ is regarded as an element of $\Isom(U, m(U))$. Note that the map $\psi_r$ is canonical, in the sense that it does not rely on any choices.

\begin{proposition} \label{prop:structure}
We have the following:
\begin{enumerate}
\item $\psi_r$ is an algebra homomorphism.
\item $\psi_r$ induces an algebra isomorphism $\fa_{n,r}/\fa_{n,r-1} \to k[\cG_{n.r}]$.
\item The $\psi_r$'s induce an algebra isomorphism
\begin{displaymath}
\psi \colon k[\fM_n] \to \prod_{r=0}^n k[\cG_{n,r}].
\end{displaymath}
\end{enumerate}
\end{proposition}

\begin{proof}
(a) One can prove this directly, but we find the following argument to be more clear. Let $m \in \fM_n$ and let $f \colon U \to V$ be an isomorphism in $\cG_{n,r}$. Let $mf \colon U \to m(V)$ be the composite map. If $m(V)$ has dimension $r$ then we define $[m] \cdot [f]=[mf]$; otherwise, we define $[m] \cdot [f]=0$. One readily verifies that this gives $k[\cG_{n,r}]$ the structure of a left $k[\fM_n]$-module. Now, for $a \in k[\fM_n]$, left multiplication by $a$ is an endomorphism of $k[\cG_{n,r}]$ as a right module over itself, and thus given by left multiplication by some element $\phi(a) \in k[\cG_{n,r}]$. Clearly, $\phi \colon k[\fM_n] \to k[\cG_{n,r}]$ is an algebra homomorphism. We have $\phi(a) = a \cdot 1$, where $1$ is the identity of $k[\cG_{n,r}]$. Since $1=\sum_U [\id_U]$, we see that $\phi([m]) = m \cdot 1 = \psi_r([m])$, which shows that $\psi_r$ is an algebra homomorphism.

(b) This argument is a slight modification of the one from \cite[\S 3]{Kovacs}. Let $R=\fa_{n,r}/\fa_{n,r-1}$, which is a unital ring (by Kov\'acs' theorem). Let
\begin{displaymath}
e_{n,r} = \left(\begin{matrix} I_r \\ & 0_{n-r} \end{matrix}\right),
\end{displaymath}
regarded as an element of $\fM_n$, and let $e=[e_{n,r}]$, which is an idempotent of $R$. One easily sees that $R=ReR$, and so, by general ring theory \cite[Theorem~4.13]{Steinberg}, $R$ is isomorphic to the endomorphism ring of $Re$, regarded as a right $eRe$-module; explicitly, this isomorphism is
\begin{displaymath}
i \colon R \to \End_{eRe}(Re), \qquad i(a)(x) = ax.
\end{displaymath}
Now, $eRe$ is naturally identified with $k[\fG_r]$, and $Re$ is identified with $k[X]$, where $X$ is the set of injective linear maps $g \colon \bF^r \to \bF^n$. The action of $R$ on $k[X]$ is similar to what was used in the previous proof: $[m] \cdot [g]$ is defined to be $[mg]$ if $mg$ is injective, and~0 otherwise. We similarly define a left $k[\cG_{n,r}]$-module structure on $k[X]$, by declaring $[f] \cdot [g]$ to be $[fg]$ if $fg$ is defined and injective, and~0 otherwise. Now, consider the diagram
\begin{displaymath}
\xymatrix@C=5em{
R \ar[r]^{\psi_r} \ar[rd]_i & k[\cG_{n,r}] \ar[d] \\
& \End_{k[\fG_r]}(k[X]) }
\end{displaymath}
where the vertical map comes from the module structure and is defined similarly to $i$. One readily verifies that this diagram commutes, that is, for $m \in \fM_n$ and $g \in X$, we have $[m] \cdot [g] = \psi_r([m]) \cdot [g]$. Since $i$ is an isomorphism, it follows that $\psi_r$ is injective. But all spaces above have the same dimension, and so $\psi_r$ is an isomorphism, as claimed.

(c) Let $\epsilon_r$ be the unit of $\fa_{n,r}$, and put $\epsilon_{-1}=0$. Then $\psi_r(\epsilon_s)$ is~1 if $s \ge r$, and~0 otherwise. We thus see that $\psi_s(\epsilon_r-\epsilon_{r-1})=\delta_{r,s}$. Now, let $x=(x_0, \ldots, x_n)$ in the target be given. Since $\psi_r$ is surjective, there is $y_r \in k[\fM_n]$ such that $x_r=\psi_r(y_r)$. Putting $y=\sum_{r=0}^n (\epsilon_r-\epsilon_{r-1}) y_r$, we have $\psi_r(y)=x_r$ for each $0 \le r \le n$, and so $\psi(y)=x$. Thus $\psi$ is surjective. Since the source and target have the same dimension, the result follows.
\end{proof}

\subsection{Representation theory of $\fG_n$} \label{ss:repGLn}

We recall some of the representation theory of $\fG_n$. Throughout, $k$ denotes an algebraically closed field of characteristic~0, and all representations are over $k$.

Let $\fP_{n,m} \subset \fG_{n+m}$ be the standard parabolic subgroup with Levi factor $\fG_n \times \fG_m$. If $\pi$ and $\sigma$ are representations of $\fG_n$ and $\fG_m$, put
\begin{displaymath}
\pi \odot \sigma = \Ind_{\fP_{n,m}}^{\fG_{n+m}}(\pi \boxtimes \sigma),
\end{displaymath}
where $\fP_{n,m}$ acts on $\pi \boxtimes \sigma$ through its Levi. This is called the \defn{parabolic induction} of $V$ and $W$. This operation defines a monoidal structure on the category
\begin{displaymath}
\bigoplus_{n \ge 0} \Rep(\fG_n).
\end{displaymath}
Moreover, this monoidal structure is braided \cite{JoyalStreet}. As a consequence, we see that (up to isomorphism), $\odot$ is commutative and associative.

A representation $\sigma$ of $\fG_n$ is called \defn{cuspidal} if it does not occur as a constituent in $\pi \odot \pi'$, where $\pi$ and $\pi'$ are representations of $\fG_r$ and $\fG_s$ with $r+s=n$ and $r,s>0$. Let $\sigma$ be a cuspidal representation of $\fG_m$. Given a partition $\lambda$ of $n$, there is an irreducible representation $\sigma[\lambda]$ of $\fG_{nm}$, which occurs in $\sigma^{\odot n}$. See \cite[Appendix B.2.2]{GH2}.

Let $\tau$ be the trivial representation of $\fG_1$. The representations $\tau[\lambda]$, as $\lambda$ varies, are called the \defn{unipotent} representations of $\fG_n$. If $\lambda=(n)$ is a single row partition then $\tau[\lambda]$ is the trivial representation of $\fG_n$; we write $\tau(n)$ for this representation. If $\lambda=(1^n)$ is a single column then $\tau[\lambda]$ is the Steinberg representation of $\fG_n$. (This convention agrees with Gurevich--Howe \cite{GH1, GH2} and Green \cite{Green}, but other references use the transposed convention.)

Let $\pi$ be an irreducible representation of $\fG_n$. There is then an isomorphism
\begin{displaymath}
\pi = \sigma_1[\lambda_1] \odot \cdots \odot \sigma_r[\lambda_r],
\end{displaymath}
where the $\sigma_i$ are non-isomorphic cuspidal representations and the $\lambda_i$ are non-empty partitions. This decomposition is unique, up to permuting the factors, see \cite[Appendix B]{GH2} and the references within or \cite[Corollary 5.14]{Carmon}.  The multiset consisting of $\vert \lambda_1 \vert$ copies of $\sigma_1$ , $\ldots$, $\vert \lambda_r \vert$ copies of $\sigma_r$ is called the \defn{cuspidal support} of $\pi$. We say that a representation $\pi$ is \defn{pure} if $\tau$ is not in its cuspidal support; this is not standard terminology. We define $\pi\{\lambda\}=\pi \odot \tau[\lambda]$. Thus every irreducible representation of $\fG_n$ has the form $\pi\{\lambda\}$ for a unique pure representation $\pi$ and (possibly empty) partition $\lambda$.

Let $\lambda$ and $\mu$ be partitions. We write $\mu \subset \lambda$ to indicate that the Young diagram for $\mu$ is contained in that for $\lambda$. We say that $\lambda/\mu$ is a horizontal strip if $\mu\subset \lambda$, and no two boxes in $\lambda \setminus \mu$ belong to the same column; we write $\lambda/\mu \in \HS$ to indicate this. We write $\lambda/\mu \in \HS_n$ if additional $\lambda$ has exactly $n$ more boxes than $\mu$.

Recall the classical Pieri rule. Let $S^{\lambda}$ denote the Specht module for the symmetric group $\fS_n$, where $\vert \lambda \vert =n$. Then
\begin{displaymath}
\Ind_{\fS_n \times \fS_m}^{\fS_{n+m}}(S^{\lambda} \boxtimes \bone) = \bigoplus_{\mu/\lambda \in \HS_m} S^{\mu}.
\end{displaymath}
We will require an analogous rule in the representation theory of $\fG_n$:

\begin{proposition}[{\cite[Equation (5.21)]{GH2}}]
Let $\pi$ be a pure representation and let $\lambda$ be a partition. Then
\begin{displaymath}
\pi\{\lambda\} \odot \tau(n) = \bigoplus_{\mu/\lambda \in \HS_n} \pi\{\mu\}.
\end{displaymath}
\end{proposition}

\begin{remark}
Note that the proposition exactly determines the decomposition of parabolic inductions of the form $\Ind_{\fP_{n,m}}^{\fG_{n+m}}(\pi \boxtimes \bone)$, where $\pi$ is an irreducible representation of $\fG_n$.
\end{remark}

\subsection{$q$-binomial coefficients} \label{ss:qbinom}

We recall some definitions and simple results concerning $q$-binomial coefficients. In our applications, $q$ will be a prime power, but in what follows $q$ can be an indeterminate. For a natural number $n$, define the \defn{$q$-integer} by
\begin{displaymath}
[n]_q = \frac{q^n-1}{q-1} = 1+q+\cdots+q^{n-1},
\end{displaymath}
and the \defn{$q$-factorial} by
\begin{displaymath}
[n]_q! = [n]_q \cdot [n-1]_q \cdots [1]_q.
\end{displaymath}
The \defn{$q$-binomial coefficient} is defined by
\begin{displaymath}
\qbinom{n}{m} = \frac{[n]_q!}{[m]_q! \cdot [n-m]_q!}.
\end{displaymath}
This is a polynomial in $q$ with $\bZ$ coefficients. When $q$ is the cardinality of a finite field $\bF$, 
this coefficient counts the number of $m$-dimensional subspaces of $\bF^n$. There is a $q$-analog of the Pascal identity:
\begin{displaymath}
\qbinom{n}{m} - \qbinom{n-1}{m-1} = q^m \qbinom{n-1}{m}.
\end{displaymath}
There is also a $q$-analog of the binomial theorem:
\begin{displaymath}
\prod_{j=0}^{n-1} (1+q^j t) = \sum_{j=0}^n q^{\binom{j}{2}} \qbinom{n}{j} t^j
\end{displaymath}

\section{Semi-idempotent matrices} \label{s:semi}

\subsection{Overview}

Let $V$ be a finite dimensional $\bF$-vector space. A linear operator $T \colon V \to V$ is \defn{semi-idempotent} if there is a decomposition $V=X \oplus Y$ into $T$-stable subspaces such that $T \restriction_X$ is the identity and $T \restriction_Y$ is nilpotent. The spaces $X$ and $Y$ are then unique: $X$ is the 1-eigenspace of $T$ and $Y$ is the generalized 0-eigenspace. Equivalently, $T$ is semi-idempotent if $T$ acts by the identity on the image of $T^n$ for $n \gg 0$. We also say that a matrix is semi-idempotent if the corresponding operator is.

As we have discussed, Kov\'acs proved that the ideal $\fa_{n,r}$ of $k[\fM_n]$ has a unit. He actually proved more: he showed that the unit is a linear combination of semi-idempotent matrices. This is why these matrices will be important for us.

In this section, we study semi-idempotent matrices in detail. The culmination of this work is Proposition~\ref{prop:hk-independent-of-k}, which is the key result needed to prove our unit formula.

\subsection{Partial linear maps}

Let $V$ be a finite-dimensional linear space over $\bF$. By a \emph{partial linear map} on $V$ we mean a linear map $T \colon W \to V$, where $W \subset V$ is a linear subspace. In this case, we denote $\dom T = W$ and call $W$ \emph{the the domain of} $T$. If $T$ and $S$ are partial linear maps on $V$, we define their composition $T \circ S$ to be the partial linear map on $V$ whose domain is $\dom\left(T \circ S\right) = S^{-1}\left(\dom T\right)$ and is defined by the formula $\left(T \circ S\right)\left(v\right) = T\left(S\left(v\right)\right)$ for any $v \in \dom\left(T \circ S\right)$. If $T$ is a partial linear map on $V$, we denote
\begin{displaymath}
T^j = \underbrace{T \circ \dots \circ T}_{j \text{ times}}
\end{displaymath}
for any $j \ge 1$ and define $T^0 = \id_{V} \colon V \to V$. We warn the reader that $T^j$ will usually not have the same domain as $T$.

We say that a partial linear map $T$ is \emph{semi-idempotent} if there exists $N \ge 0$ such that $T$ acts trivially on $\im T^{N}$. In this case, we have that $T^{n} = T^{N}$ for any $n \ge N$ and we denote $T^{\infty} = T^N$. 

Suppose that $T \colon V \to V$ is semi-idempotent. Notice that since $T^{N}\left(T-1\right) = 0$, we have that the only possible eigenvalues of $T$ are $0$ and $1$ and that the algebraic multiplicity of $1$ with respect to $T$ equals its geometric multiplicity. Conversely, it is easy to show that if $T \colon V \to V$ has only eigenvalues $\in \left\{0,1\right\}$ and the algebraic multiplicity of $1$ with respect to $T$ equals its geometric multiplicity, then $T$ is semi-idempotent.

\subsection{Counting semi-idempotents}\label{ss:counting-semi-idempotents}

For $\ell \ge 1$ let
\begin{displaymath}
J_{\ell}\left(0\right) = \begin{pmatrix}
	0 & \\
	1 & 0 & \\
	& 1 & \ddots \\
	& & \ddots & 0 & \\
	& & & 1 & 0
\end{pmatrix}
\end{displaymath}
be the lower $\ell \times \ell$ Jordan matrix with eigenvalue $0$. Let $b_1,\dots,b_n$ be the standard basis of $\bF^n$.

\begin{lemma} \label{lem:extensions-of-jordan-matrix}
Let $n = r_1 + \dots + r_{s}$ for some $r_1,\dots,r_{s} > 0$, and let
\begin{displaymath}
A = \diag\left(J_{r_1}\left(0\right),\dots,J_{r_{s}}\left(0\right)\right).
\end{displaymath}
Then
\begin{enumerate}
\item The number of nilpotent matrices that can be obtained by changing the last column of $A$ equals the number of non-nilpotent semi-idempotent matrices that can be obtained by changing the last column of $A$. Both of these numbers equal $q^{n - r_{s}}$.
\item If $r_{s} > 1$ then the number of nilpotent matrices with rank equal to $\rank A$ that can be obtained by changing the last column of $A$ equals the number of non-nilpotent semi-idempotent matrices with rank equal to $\rank A$ that can be obtained by changing the last column of $A$, and both of these numbers equal $q^{n - r_{s} - \left(s - 1\right)}$.
\end{enumerate}
\end{lemma}

\begin{proof}
If $B$ is a semi-idempotent matrix obtained by changing the last column of $A$ then the space consisting of the fixed vectors of $B$ is at most one-dimensional. Indeed, otherwise there would be a fixed vector of $A$ lying in $\Span_{\bF}\left(b_1,\dots,b_{n-1}\right)$, which is not true. Thus from the discussion above, the characteristic polynomial of $B$ has to be either $X^n$ or $X^{n-1}\left(X-1\right)$.

Conversely, if $B$ has characteristic polynomial $X^n$ or $X^{n-1}\left(X-1\right)$ then the only eigenvalues of $B$ are $0$ and $1$ and the algebraic multiplicity of $1$ in $B$ is the same as the geometric multiplicity of $1$ in $B$ and thus from the discussion above $B$ is semi-idempotent. It therefore suffices to examine the characteristic polynomial of a matrix $B$ obtained by changing the last column of $A$.

Let $B$ be the matrix obtained by replacing the last column of $A$ with the column given by $\transpose{\left(a_1,\dots,a_n\right)}$, where $a_1,\dots,a_n \in \bF$. Then $B$ is an upper triangular block matrix. Let us denote the last diagonal block by $B'$, that is, $B'$ is the $r_{s} \times r_{s}$ bottom right block of $B$. Then
\begin{displaymath}
B' = \begin{pmatrix}
	0 & 0 & 0 & \dots & 0 & a_{n - r_{s} + 1} \\
	1 & 0 & 0 & \dots & 0 & a_{n - r_{s} + 2} \\
	0 & 1 & 0 & \dots & 0 & a_{n - r_{s} + 3} \\
	\vdots & \vdots & \ddots & \ddots & \vdots & \vdots \\ 
	0 & 0 & \dots &  1 & 0 & a_{n - 1} \\
	0 & 0 & 0 & \dots & 1 & a_{n}
\end{pmatrix}.
\end{displaymath}
The characteristic polynomial of $B$ is the product of the characteristic polynomials of $J_{r_1}\left(0\right)$, $\dots$, $J_{r_{s - 1}}\left(0\right)$ and of the characteristic polynomial of $B'$. Notice that $B'$ is a companion matrix of a polynomial. Thus the characteristic polynomial of $B$ is
\begin{displaymath}
X^{n - r_{s}} \left(X^{r_{s}} - \sum_{i=1}^{r_{s}} a_{n - r_{s} + i} X^{i-1}\right).
\end{displaymath}
Therefore $B$ will be a nilpotent matrix if and only if $a_{n-r_{s} + i} = 0$ for every $1 \le i \le r_{s}$, and $B$ will be a non-nilpotent semi-idempotent matrix if and only if $a_n = 1$ and $a_{n-r_{s} + i} = 0$  for every $1 \le i \le s - 1$. In both cases we have $q^{n-r_{s}}$ options for $B$. This finishes the proof of the first part.

We move to count matrices $B$ with the same rank as $A$. In the nilpotent matrix case, $B$ will have the same rank as of $A$ if and only if \begin{equation}\label{eq:condition-for-rank-of-C-to-remain-the-same-as-rank-of-A}
	a_1 = a_{r_1 + 1} = \dots = a_{r_{s - 2} + 1} = 0.
\end{equation} 
In the non-nilpotent case, since $r_{s} > 1$, $B$ has $b_n$ as its $n-1$ column and therefore $B$ will have the same rank as of $A$ if and only if condition \eqref{eq:condition-for-rank-of-C-to-remain-the-same-as-rank-of-A} holds. Thus in both cases we have $q^{n-r_{s} - \left(s - 1\right)}$ options for $B$.
\end{proof}

We define the \defn{stable rank} of a matrix $m$, denoted $\srk{m}$, to be the rank of $m^n$ for $n \gg 0$.

\begin{lemma}\label{lem:extension-of-semi-idempotent-jordan-matrix}
Let
\begin{displaymath}
A = \diag\left(J_{r_1}\left(0\right),\dots,J_{r_{s}}\left(0\right)\right)
\end{displaymath}
as above. Let $A' = \diag\left(I_{\ell}, A\right)$ for some $\ell \ge 0$.
\begin{enumerate}
\item For $j \ge 0$ let $e\left(A', j\right)$ be the number of semi-idempotent $C$ matrices that can be obtained by changing the last column of $A'$ such that $\srk{C} = j$. Then
\begin{displaymath}
e\left(A',\srk{A'} \right) = q^{\srk{A'}} e\left(A',\srk{A'} + 1\right).
\end{displaymath}
\item For $i,j \ge 0$ let $e\left(A', i,j\right)$ be the number of semi-idempotent $C$ matrices that can be obtained by changing the last column of $A'$ such that $\rank C = i$ and $\srk{C} = j$. If $s > 0$ and $r_{s} > 1$ then
\begin{displaymath}
e\left(A', \rank A', \srk{A'} \right) = q^{\srk{A'}} e\left(A', \rank A', \srk{A'} + 1\right).
\end{displaymath}
\end{enumerate}
\end{lemma}

\begin{proof}
Suppose that $C$ is a matrix such that $\srk{C} = \srk{A'} + 1$. Then there exists a non-zero fixed $C$-fixed vector $v = \sum_{i=1}^{\ell+n} \alpha_i b_i \notin \Span_{\bF}\left(b_1,\dots,b_{\ell+n-1}\right)$ where $\alpha_1,\dots,\alpha_{\ell+n} \in \bF$. Since $v \notin \Span_{\bF}\left(b_1,\dots,b_{\ell+n-1}\right)$, we may assume without loss of generality that $\alpha_{\ell+n} = 1$. Then $Cv = v$, which implies
\begin{displaymath}
Cb_{\ell+n} = b_{\ell+n} + \sum_{j=\ell+1}^{\ell+n-1}\alpha_j \left(b_j - \diag\left(I_{\ell}, A\right) b_j\right).
\end{displaymath}
Thus we must have that $C = \diag\left(I_{\ell}, B\right)$, where $B$ is a semi-idempotent matrix with a non-zero fixed vector obtained by changing the last column of $A$ as in Lemma~\ref{lem:extensions-of-jordan-matrix}. Conversely, it is clear that if $B$ is a semi-idempotent matrix with a fixed vector as in Lemma~\ref{lem:extensions-of-jordan-matrix} then $C = \diag\left(I_{\ell}, B\right)$ is a semi-idempotent matrix obtained by changing the last column of $A'$ and that $\srk{C} = \ell + 1 = \srk{A'} + 1$. Thus we see that the number of semi-idempotent matrices $C$ obtained from changing the last column of $A$ that satisfy $\srk{C} = \srk{A'} + 1$ is $q^{n-r_{s}}$ as given in Lemma~\ref{lem:extensions-of-jordan-matrix}. If $r_{s} > 1$, then by Lemma \ref{lem:extensions-of-jordan-matrix} we have that there are $q^{n-r_{s}-\left(s-1\right)}$ choices of $B$ with $\rank B = \rank A$ and it is clear that for each such $B$ we have that $C = \diag\left(I_{\ell}, B\right)$ satisfies $\rank C = \rank A'$.

We move to count the number of semi-idempotent matrices $C$ obtained by changing the last column of $A'$ such that $\srk{C} = \srk{A'} = \ell$. Let $C$ be the matrix obtained by replacing the last column of $A'$ with the column $\transpose{\left(a_{-\left(\ell-1\right)},a_{-\left(\ell-2\right)},\dots,a_{0},a_1,\dots,a_n\right)}$, where $a_{-\left(\ell-1\right)},\dots,a_n \in \bF$. As in the proof of Lemma \ref{lem:extensions-of-jordan-matrix}, we have that $C$ is a semi-idempotent matrix such that $\srk{C} = \srk{A'}$ if and only if the characteristic polynomial of $C$ is the characteristic polynomial of $A'$ which is $X^n \left(X-1\right)^{\ell}$. As in the proof of Lemma~\ref{lem:extensions-of-jordan-matrix}, we have that the characteristic polynomial of $C$ is
\begin{displaymath}
\left(X-1\right)^{\ell} X^{n-r_{s}} \left(X^{r_{s}} - \sum_{i=1}^{r_{s}} a_{n - r_{s} + i} X^{i-1}\right)
\end{displaymath}
and thus $C$ is a matrix of the desired form if and only if we have $a_{n-r_{s} + i} = 0$ for every $1 \le i \le r_{s}$. Therefore there are $q^{\ell + n - r_{s}}$ options for semi-idempotent matrices $C$ obtained from changing the last column such that $\srk{C} = \srk{A'}$. As in the proof of Lemma~\ref{lem:extensions-of-jordan-matrix}, a matrix $C$ as above satisfies $\rank C = \rank A'$ if and only if $a_1 = a_{r_1+1} = \dots = a_{r_{s - 2} + 1} = 0$ and thus there are $q^{\ell + n - r_{s} - \left(s - 1\right)}$ options for $C$ satisfying $\rank C = \rank A'$.
\end{proof}

Next, we will use results from \cite[Section 5.6]{Steinberg} to obtain the following coordinate-free lemma.

\begin{lemma} \label{lem:coordinate-free-extension-of-corank-1-partial-semi-idempotent}
Suppose that $T$ is a partial linear map with domain $W \subset V$ of co-dimension~1.
\begin{enumerate}
\item \label{item:identity-for-number-of-matrices-with-given-rank-at-infinity} Let
\begin{displaymath}
e\left(T,j\right) = \#\left\{ S \colon V \to V \text{ semi-idempotent} \mid S\restriction_{W} = T,  \srk{S} = j \right\}.
\end{displaymath}
Then we have the following equality
\begin{equation*}
e\left(T,\srk{T} \right) = q^{\srk{T}} \cdot e\left(T, \srk{T} + 1\right).
\end{equation*}
\item \label{item:identity-for-number-of-matrices-with-given-rank-at-infinity-and-same-rank} Let
\begin{displaymath}
e\left(T,i,j\right) = \#\left\{ S \colon V \to V \text{ semi-idempotent} \mid S\restriction_{W} = T, \rank S = i, \srk{S} = j \right\}.
\end{displaymath}
If $\im T \not\subset W$, then we have the following equalities
\begin{equation*}
e\left(T, \rank T, \srk{T} \right) = q^{\srk{T}} \cdot e\left(T, \rank T, \srk{T} + 1\right),
\end{equation*}
and
\begin{equation*}
e\left(T, \rank T + 1, \srk{T} \right) = q^{\srk{T}} \cdot e\left(T, \rank T + 1, \srk{T} + 1\right).
\end{equation*}
\end{enumerate}
\end{lemma}

Before the proof, we recall some notations from \cite[Sections 3.3 and 5.6]{Steinberg}. A \emph{partial map} on a finite set $\cB$ is a map $f \colon \dom f \to \cB$, for some subset $\dom f \subset \cB$, which we call the \emph{domain} of $f$. A \emph{partial injective map} on $\cB$ is an injective map $f \colon \dom f \to \cB$ as above.

Given partial maps $f$ and $g$ on $\cB$, we define their composition $f \circ g$ as the partial map on $\cB$ with domain $\dom\left(f\circ g\right) = f^{-1}\left(\dom g\right)$ given by the rule $\left(f \circ g\right)\left(b\right) = f\left(g\left(b\right)\right)$ for any $b \in \cB$. We define $f^0 = \id_{\cB} \colon \cB \to \cB$ and $f^j = \underbrace{f \circ \dots \circ f}_{j \text{ times}}$ for $j \ge 1$. We say that $f$ is \emph{semi-idempotent} if it acts trivially on $\im f^{\infty} = \bigcap_{j=0}^{\infty} \im f^{j}$.

Suppose now that $V$ is a finite dimensional vector space $V$ over $\bF$ and let $\cB$ be a basis of $V$. Given a partial function $f$ on $\cB$, we define a partial linear map $T_f$ on $V$ as follows. Its domain is $\dom T_f = \Span_{\bF}\left(\dom f\right)$ and it is defined by $T_f b = f\left(b\right)$ for any $b \in \dom f$. We also define a linear map $S_f \colon V \to V$ with $S_f\restriction_{\dom T_f} = T_f$ by extending $T_f$ to be zero on the remaining basis elements in $\cB \setminus \dom f$.

\begin{proof}[Proof of Lemma \ref{lem:coordinate-free-extension-of-corank-1-partial-semi-idempotent}]
We choose an arbitrary decomposition $W = U \oplus \ker T$. Then the restriction $T \restriction_U$ is injective. By \cite[Lemma 5.40]{Steinberg} we can find a basis $\cB$ of $V$ and a partial injective semi-idempotent map $f$ on $\cB$ such that $T \restriction_U = T_f$. We may choose $\cB$ so that its first $\dim U$ elements lie in $U$, the next $\dim \ker T$ elements lie in $V \setminus U$, and the last element of $\cB$ lies in $V \setminus W$. We define $S = S_f \colon V \to V$. By \cite[Lemma 5.36]{Steinberg}, $\rank S^{\infty} = \rank T^{\infty}$. There exists a permutation $\cB'$ of $\cB$ such that the matrix representation of $S$ with respect to $\cB'$ is of the form $A'$ as in Lemma \ref{lem:extension-of-semi-idempotent-jordan-matrix} with $1 \le r_1, \dots, r_{s}$, and such that the last element of $\cB'$ still lies in $V \setminus W$. Using the last element of $\cB'$ we replace the elements that lie in $V \setminus U$ with elements of $\ker T$, except for the last element which we keep unchanged. Denote the modified basis by $\cB''$. The matrix representation of $S$ with respect to $\cB''$, which we denote $A''$, is the same as $A'$, except for the last row that might now have entries before the first $1$ value of the last row of $A'$ (if such entry exists). We have the equality $A'' = A' \cdot \left(\begin{smallmatrix}
	u\\
	& 1
\end{smallmatrix}\right)$, where $u \in \GL_{\ell+n-1}\left(\bF\right)$ is a lower unipotent matrix with zeros everywhere outside of its diagonal, except for the last row which consists of the first $\ell+n-1$ coordinates of the last row of $A''$.

For a matrix $B \in \rM_{\ell+n}\left(\bF\right)$ we have that the linear map represented by $B$ with respect to the basis $\cB''$ agrees with $T$ on $W$ if and only if $A'' \left(\begin{smallmatrix}
I_{n-1}\\
& 0
\end{smallmatrix}\right) = B \left(\begin{smallmatrix}
I_{n-1}\\
& 0
\end{smallmatrix}\right)$ which is equivalent to $A' \left(\begin{smallmatrix}
u\\
& 0
\end{smallmatrix}\right) = B \left(\begin{smallmatrix}
I_{n-1}\\
& 0
\end{smallmatrix}\right)$ which in turn is equivalent to $A' \left(\begin{smallmatrix}
I_{n-1}\\
& 0
\end{smallmatrix}\right) = B \left(\begin{smallmatrix}
u^{-1} \\
& 1
\end{smallmatrix}\right) \left(\begin{smallmatrix}
I_{n-1}\\
& 0
\end{smallmatrix}\right)$. Therefore, we have $e\left(T, j\right) = e\left(A', j\right)$ and $e\left(T, i, j\right) = e\left(A', i, j\right)$, where $A'$ is the matrix representing $S$ in basis $\cB'$. Thus the first part of the lemma follows from the first part of Lemma \ref{lem:extension-of-semi-idempotent-jordan-matrix}.

Regarding the second part of the lemma, if $\im T \not\subset W$ then the matrix representation of $S$ with respect to basis $\cB''$, which we denoted $A''$, has a non-zero entry in its last row and thus so does $A'$. Therefore in the notation of Lemma \ref{lem:extension-of-semi-idempotent-jordan-matrix}, we have that $r_{s} > 1$, which from the second part of Lemma \ref{lem:extension-of-semi-idempotent-jordan-matrix} implies that \begin{equation*}
e\left(A', \rank A', \srk{A'} \right) = q^{\srk{A'}} \cdot e\left(A', \rank A', \srk{A'} + 1\right).
\end{equation*}
Thus the first equality in part (\ref{item:identity-for-number-of-matrices-with-given-rank-at-infinity-and-same-rank}) now follows. The second equality in part (\ref{item:identity-for-number-of-matrices-with-given-rank-at-infinity-and-same-rank}) follows from subtracting the equality of part (\ref{item:identity-for-number-of-matrices-with-given-rank-at-infinity}) from the first equality of part (\ref{item:identity-for-number-of-matrices-with-given-rank-at-infinity-and-same-rank}).
\end{proof}

\subsection{Some semi-idempotent sums}\label{ss:some-semi-idempotent-sums}

Put
\begin{displaymath}
\mu(n) = (-1)^n q^{\binom{n}{2}}.
\end{displaymath}
For every $0 \le j \le n$, let $V_j = \Span_{\bF}\left(b_1,\dots,b_j\right)$. Throughout this subsection we will identify partial linear maps $T \colon V_j \to V_n$ (with domain $V_j$) with $n \times n$ matrices with zeros at their $j+1,j+2,\dots,n$ columns.

The next two propositions allow us to reduce certain sums over semi-idempotent maps that will appear in the proof of Theorem~\ref{thm:unit}, and they will be used repeatedly in the proof. The first one tells us that for certain sums over semi-idempotent elements $T \colon V_{n-1} \to V_n$ we only need to consider semi-idempotent elements of the form $T \colon V_{n-1} \to V_{n-1}$.

\begin{proposition} \label{prop:vanishing-identity-of-coefficient-of-semi-idempotent}
Let $T \colon V_{n-1} \to V_n$ be a semi-idempotent partial linear map with domain $V_{n-1}$ such that $\im T \not\subset V_{n-1}$, and let $r \ge 0$. Then for any function $f \colon \bZ_{\ge 0} \to \bZ$, we have 
\begin{equation}\label{eq:sum-of-semi-idempotent-coefficient-for-small-rank}
\sum_{\substack{S \colon V_n \to V_n\\
		S \text{ semi-idempotent}\\
		\rank S \le r\\
		S \restriction_{V_{n-1}} = T}}
\mu(\srk{S}) f\left(\rank S\right) = 0.
\end{equation}
\end{proposition}

\begin{proof}
Let us denote $\ell = \srk{T}$. For $S$ as in the sum we must have either $\srk{S} = \ell$ or $\srk{S} = \ell + 1$. Otherwise there would be a non-zero fixed vector of $S$ lying in $V_{n-1}$ that is not a fixed vector of $T$, which is impossible. We may as well assume $\rank T \le r$, otherwise both sides of the equation vanish. If $\rank T \le r - 1$, then the rank of $S$ must be either $\rank T$ or $\rank T + 1$. If $\rank T = r$, then the rank of $S$ must be $\rank T$. Thus the sum \eqref{eq:sum-of-semi-idempotent-coefficient-for-small-rank} can be rewritten as
\begin{displaymath}
\sum_{j=\rank T}^{\min\left(r, \rank T + 1\right)} \left(-1\right)^{\ell} \left(q^{\binom{\ell}{2}} e\left(T, j, \ell\right) - q^{\binom{\ell + 1}{2}} e\left(T, j, \ell + 1\right)\right) f\left(j\right)
\end{displaymath}
Since $\binom{\ell+1}{2} = \binom{\ell}{2} + \ell$, by Lemma~\ref{lem:coordinate-free-extension-of-corank-1-partial-semi-idempotent} we have the equality 
\begin{displaymath}
q^{\binom{\ell}{2}} e\left(T,j,\ell\right) = q^{\binom{\ell+1}{2}} e\left(T,j,\ell+1\right),
\end{displaymath}
which implies the desired identity.
\end{proof}

The next proposition allows us to evaluate certain sums corresponding to extensions of semi-idempotent linear maps $T \colon V_{n-1} \to V_{n-1}$.

\begin{proposition}\label{prop:reduction-of-unit-inner-sum}
Let $T \colon V_{n-1} \to V_{n-1}$ be a semi-idempotent element of rank $\le r$. For any function $f \colon \bZ_{\ge 0} \to \bZ$ denote 
\begin{displaymath}
E_n\left(f,T\right) = \sum_{\substack{S \colon V_n \to V_n\\
	S \text{ semi-idempotent}\\
	\rank S \le r\\
	S \restriction_{V_{n-1}} = T}} \mu(\srk{S}) f\left(\rank S\right).
\end{displaymath}
\begin{enumerate}
\item If $\rank T \le r-1$ then
\begin{displaymath}
E_n\left(f, T\right) = \mu(\srk{T}) q^{\rank T} \left(f\left(\rank T\right) - f\left(\rank T + 1\right)\right).
\end{displaymath}
\item If $\rank T = r$ then
\begin{displaymath}
E_n\left(f,T\right) =  \mu(\srk{T}) \cdot  \left(q^{\rank T} f\left(\rank T\right)\right).
\end{displaymath}
\end{enumerate}
\end{proposition}

\begin{proof}
Suppose that $\rank T \le r-1$. Then $S$ in the sum must satisfy $\rank S = \rank T$ or $\rank S = \rank T + 1$, and in the latter case $S$ must satisfy $\srk S = \srk T$ or $\srk S = \srk T + 1$. We analyze the various cases, noting that $S$ is completely determined by $Sb_n$.
\begin{itemize}
\item Suppose $\rank S = \rank T$. Then, by characteristic polynomial considerations, $\srk{T} = \srk{S}$. Since $S$ has the same image as $T$, we see that $S b_n$ must belong $\im T$. Moreover, any such choice for $Sb_n$ yields an $S$ as in the sum. Thus there are $\left|\im T\right| = q^{\rank T}$ options for such $S$.
\item Suppose $\rank S = \rank T + 1$ and $\srk S = \srk T$. Looking at characteristic polynomials, we see that the final coordinate of $S b_n$ must vanish. Provided that we choose $Sb_n$ so that it is outside the image of $T$ and satisfies this condition, we obtain an $S$ as in the sum. We thus obtain $q^{n-1} - q^{\rank T}$ possibilities in this case.
\item Finally suppose $\rank S = \rank T + 1$ and $\srk S = \srk T+1$. Choosing a basis $u_1,\dots,u_{n-1}$ for $V_{n-1}$ such that its first $\srk T$ elements are fixed vectors of $T$ and setting $u_n = b_n$, we have that there are $q^{n-1-\srk{T}}$ options for $S b_n$. Explicitly, every choice of scalars $\underline{\alpha} = \left(\alpha_1, \dots, \alpha_{n-1-\srk T}\right) \in \bF^{n-1-\srk T}$ corresponds to a linear map $S_{\underline{\alpha}}$ extending $T$ such that $S_{\underline{\alpha}} v_{\underline{\alpha}} = v_{\underline{\alpha}}$, where \begin{displaymath}
	v_{\underline{\alpha}} = b_n + \sum_{i=1}^{n-\srk T - 1} \alpha_i u_{i + \srk T}
\end{displaymath} and these are all the possible extensions $S$ of $T$ with $\srk S = \srk T + 1$. The proof of this is similar to the argument at the beginning of the proof of Lemma \ref{lem:extension-of-semi-idempotent-jordan-matrix}.
\end{itemize}
In all cases, the element $S$ we constructed is semi-idempotent. To see this, first recall that since $T$ is semi-idempotent, the space $V_{n-1}$ is a direct sum of the generalized $0$-eigenspace of $T$ and the $1$-eigenspace of $T$. Then, by our construction, the space $V_n$ is the direct sum of $V_{n-1}$ and the linear span of either a generalized $0$-eigenvector (in the first two cases) or a fixed vector (in the third case) of $S$.

Combining the above computations, we obtain
\begin{equation*}
\begin{split}
E_n\left(f,T\right) =& \left(-1\right)^{\srk{T}} q^{\binom{\srk{T}}{2}} \left(q^{\rank T} \cdot f\left(\rank T\right) + \left(q^{n-1} - q^{\rank T}\right) f\left(\rank T + 1\right)\right)\\
& + q^{n-1-\srk{T}} \cdot \left(-1\right)^{\srk{T} + 1} q^{\binom{\srk{T} + 1}{2}} f\left(\rank T + 1\right),
\end{split}
\end{equation*}
which equals
\begin{displaymath}
E_n\left(f, T\right) = \left(-1\right)^{\srk{T}} q^{\binom{\srk{T}}{2}} q^{\rank T} \left(f\left(\rank T\right) - f\left(\rank T + 1\right)\right)
\end{displaymath}
thanks to the identity $-\srk{T} + \binom{\srk{T} + 1}{2} = \binom{\srk{T}}{2}$.

The case $\rank T = r$ is simpler since we only have the first option $\rank S = \rank T = r$, and thus we get the desired result.
\end{proof}

\subsection{The key result}

Given a semi-idempotent element $T_r \colon V_r \to V_n$ and $0 \le \ell \le n-r$ and $0 \le j \le r$, put
\begin{displaymath}
h_{\ell}(j,T_r) = q^{\ell j} \sum_{\substack{S \colon V_{n-\ell} \to V_{n-\ell}\\
S \text{ semi-idempotent}\\	\rank S \le j\\
S \restriction_{V_r} = T_r}} \mu(\srk{S}) \qbinom{n-\ell-\rank S}{j-\rank S}.
\end{displaymath}
The following is the key result used in our determination of the unit.

\begin{proposition} \label{prop:hk-independent-of-k}
The quantity $h_{\ell}(j,T_r)$ is independent of $\ell$.
\end{proposition}

\begin{proof}
Let us rewrite $h_{\ell-1}\left(j, T_r\right)$ as
\begin{displaymath}
h_{\ell-1}\left(j,T_r\right) = q^{\left(\ell-1\right)j} \sum_{\substack{T \colon V_{n-\ell} \to V_{n-\ell+1}\\
T \text{ semi-idempotent}\\
\rank T \le j\\
T \restriction_{V_r} = T_r}} \sum_{\substack{S \colon V_{n-\ell+1} \to V_{n-\ell+1}\\
S \text{ semi-idempotent}\\
\rank S \le j\\
S \restriction_{V_{n-\ell}} = T}} \mu(\srk{S}) \qbinom{n-\ell+1-\rank S}{j-\rank S}.
\end{displaymath}
By Proposition~\ref{prop:vanishing-identity-of-coefficient-of-semi-idempotent}, the sum over $S$ vanishes unless $\im T \subset V_{n-\ell}$. Thus we may replace $T$ above with $T \colon V_{n-\ell} \to V_{n-\ell}$. By applying Proposition \ref{prop:reduction-of-unit-inner-sum} to this modified version of the equation, we obtain \begin{displaymath}
\begin{split}
&q^{\left(\ell-1\right)j} \sum_{\substack{T \colon V_{n-\ell} \to V_{n-\ell}\\
	T \text{ semi-idempotent}\\
	\rank T \le j-1\\
	T \restriction_{V_r} = T_r}} \mu(\srk{T}) q^{\rank T} \left( \qbinom{n-\ell+1-\rank T}{j-\rank T} - \qbinom{n-\ell-\rank T}{j-\rank T - 1} \right)\\
&+ q^{\left(\ell-1\right)j} \sum_{\substack{T \colon V_{n-\ell} \to V_{n-\ell}\\
	T \text{ semi-idempotent}\\
	\rank T = j\\
	T \restriction_{V_r} = T_r}} \mu(\srk{T}) q^{\rank T} .
\end{split}
\end{displaymath}
which by the $q$-binomial Pascal identity equals 
\begin{displaymath}
q^{\left(\ell-1\right)j} \sum_{\substack{T \colon V_{n-\ell} \to V_{n-\ell}\\
	T \text{ semi-idempotent}\\
	\rank T \le j \\
	T \restriction_{V_r} = T_r}} \mu(\srk{T}) q^{\rank T} q^{j-\rank T} \qbinom{n-\ell-\rank T}{j-\rank T} = h_{\ell}\left(j,T_r\right).
\end{displaymath}
The result thus follows by induction.
\end{proof}

\section{The unit formula} \label{s:unit}

In this section, we prove our formula for the unit of $k[\fM_n]$, and some related results. We fix $n$ throughout \S \ref{s:unit}.

\subsection{The main theorem}

Define an element of $k[\fM_n]$ by
\begin{displaymath}
\eta_r = q^{-\left(n-1\right)r} \mu(r) \sum_m \mu(\srk{m}) \qbinom{n-1-\rank m}{r - \rank m} [m],
\end{displaymath}
where the sum is over all semi-idempotent matrices $m$ of rank $\le r$. Recall that
\begin{displaymath}
\mu(r)=(-1)^r q^{\binom{r}{2}},
\end{displaymath}
and $\srk{m}$ is the stable rank of $m$, i.e., the rank of $m^r$ for all $r \gg 0$. The following is our main result on the unit of $\fa_{n,r}$:

\begin{theorem} \label{thm:unit}
The element $\eta_r$ is the unit of $\fa_{n,r}$, for any $0 \le r \le n-1$.
\end{theorem}

The general proof is given in \S \ref{ss:proof-unit}. When $r=n-1$, things simplify substantially; we give a self-contained proof of that case in \S \ref{ss:proof-unit-1}. We also prove another result about $\eta_r$, which we state here. For a subspace $W$ of $V=\bF^n$, define an element of $k[\fM_n]$ by
\begin{displaymath}
\eta_W = q^{-\dim W \dim\left(V \slash W\right)} \mu(\dim{W})^{-1}
\sum_{\substack{T \colon V \to V\\
		T \text{ semi-idempotent}\\
		\im T \subset W}} \mu(\srk{T}) \left[T\right].
\end{displaymath}
Recall the map $\psi_r \colon k[\fM_n] \to k[\cG_{n,r}]$ from \S \ref{ss:kovacs}.

\begin{theorem} \label{thm:unit-decomp}
The $\eta_W$ are orthogonal idempotents of $k[\fM_n]$, and
\begin{displaymath}
\eta_r = \sum_{\dim{W} \le r} \eta_W.
\end{displaymath}
Moreover, we have $\psi_r(\eta_W)=[\id_W]$ if $\dim{W}=r$, and $\psi_r(\eta_W)=0$ if $\dim{W}<r$.
\end{theorem}

The proof of this theorem is given in \S \ref{ss:unit-decomp}. Using the $\eta_W$'s, one can find an explicit inverse image of any element of $k[\cG_{n,r}]$ under $\psi_r$. Indeed, suppose $f \colon W \to W'$ is an isomorphism of $r$-dimensional subspaces of $\bF^n$, so that $[f]$ is a typical basis vector of $k[\cG_{n,r}]$. Let $m \in \fM_n$ be an arbitrary extension of $f$ to an endomorphism of $\bF^n$. Then $[f]=\psi_r([m] \eta_W)$.

\subsection{Proof of Theorem~\ref{thm:unit} in corank~1} \label{ss:proof-unit-1}

When $r=n-1$, the formula for $\eta_r$ simplifies somewhat:
\begin{displaymath}
	\eta_{n-1} = -\mu(n)^{-1} \sum \mu(\srk{m}) [m],
\end{displaymath}
where the sum is over singular semi-idempotent matrices $m \in \fM_n$. The proof of Theorem~\ref{thm:unit} is also fairly simple in this case, as we now explain. By \cite[Lemma 3.1]{Kuhn2} (see also Lemma~\ref{lem:kuhn-criterion-for-unit} below) it suffices to check that $\eta_{n-1} \left[e_{n,n-1}\right] = \left[e_{n,n-1}\right]$, where $e_{n,n-1} = \left(\begin{smallmatrix}
	I_{n-1}\\
	& 0
\end{smallmatrix}\right)$. Using the notation of Lemma~\ref{lem:coordinate-free-extension-of-corank-1-partial-semi-idempotent} and \S\ref{ss:some-semi-idempotent-sums}, for a linear map $T \colon V_{n} \to V_n$ such that $T b_n = 0$ and such that $T \restriction_{V_{n-1}}$ is semi-idempotent, the coefficient of $\left[T\right]$ in the linear combination $\eta_{n-1} \left[e_{n,n-1}\right]$ is
\begin{displaymath}
	-\mu\left(n\right)^{-1} \left(e\left(T \restriction_{V_{n-1}}, \srk T\right) \mu\left(\srk T\right) + e\left(T \restriction_{V_{n-1}}, \srk T + 1\right) \mu\left(\srk T + 1\right)\right)
\end{displaymath}
if $T \ne e_{n,n-1}$ and is $-\mu\left(n\right)^{-1} q^{n-1} \mu\left(n - 1\right)$ if $T = e_{n-1,n}$. By Lemma~\ref{lem:coordinate-free-extension-of-corank-1-partial-semi-idempotent}, the coefficient above is zero if $T \ne e_{n,n-1}$ because $\mu\left(\srk T + 1\right) = -q^{\srk T} \mu\left(\srk T\right)$. If $T = e_{n-1,n}$, then the coefficient of $\left[T\right]$ is $1$ because $\mu\left(n\right)^{-1} \mu\left(n-1\right) = -q^{-\left(n-1\right)}$.

\subsection{Proof of Theorem~\ref{thm:unit}} \label{ss:proof-unit}

We require a few lemmas before proving the theorem. We begin by providing an alternate formula for $\eta_r$.

\begin{lemma} \label{lem:eta-alt}
We have
\begin{displaymath}
\eta_r = \sum_{j=0}^r q^{-\left(n-j\right)j} \mu(j)^{-1} \sum_{\substack{m \in \fM_n \\
	\rank m \le j\\
	m \text{ semi-idempotent}}} \mu(\srk{m}) \qbinom{n - \rank m }{j - \rank m} [m].
\end{displaymath}
\end{lemma}

\begin{proof}
Using the $q$-binomial identity we see that
\begin{displaymath}
\eta_j - \eta_{j-1} = q^{-(n-j)j} \mu(j)^{-1} \sum_{\substack{m \in \fM_n \\
\rank m \le j\\
m \text{ semi-idempotent}}} \mu(\srk{m}) \qbinom{n - \rank m }{j - \rank m} \left[m\right].
\end{displaymath}
Here we set $\eta_{-1}=0$. Summing over $0 \le j \le r$ gives the stated formula.
\end{proof}

Put
\begin{displaymath}
e_{n,r} = \left(\begin{matrix} I_r \\ & 0_{n-r} \end{matrix}\right),
\end{displaymath}
regarded as an element of $\fM_n$. The next lemma gives a helpful characterization of the unit of $\fa_{n,r}$; it is based on \cite[Lemma 3.1]{Kuhn2}, which in turn is based on \cite{Kovacs}.

\begin{lemma} \label{lem:kuhn-criterion-for-unit}
$u \in \fa_{n,r}$ is a unit if and only if the following two conditions hold:
\begin{enumerate}
\item $\left[g^{-1}\right] \cdot u \cdot \left[g\right] = u$ holds in $k[\fM_n]$, for any $g \in \fG_n$.
\item $u \cdot [e_{n,r}] = [e_{n,r}]$.
\end{enumerate}
\end{lemma}

\begin{proof}
The proof is similar to \cite[Lemma 3.1]{Kuhn2}.

Assume that $u \in \fa_r$ is a unit. Since $e_{n,r}$ is of rank $r$ we must have that the second property holds. Regarding the first property, let $g \in \fG_n$. Conjugation by $g$ preserves the set consisting of matrices of rank $\le r$ in $\fM_n$. Hence $u \cdot \left[g m g^{-1}\right] = \left[g m g^{-1}\right]$ for any $m \in \fM_n$ with rank $\le r$. Therefore we have that $\left(\left[g^{-1}\right] \cdot u \cdot \left[g\right]\right) \cdot \left[m\right] = \left[m\right]$ for any $m \in \fM_n$ with rank $\le r$. Similarly, we have that $\left[m\right]  \cdot \left(\left[g^{-1}\right] \cdot u \cdot \left[g\right]\right) = \left[m\right]$ for any $m \in \fM_n$ with rank $\le r$. Thus we have that $\left[g^{-1}\right] \cdot u \cdot \left[g\right]$ is also a unit of $\fa_{n,r}$. By the uniqueness of a unit in an algebra we must have $\left[g^{-1}\right] \cdot u \cdot \left[g\right] = u$.

Assume now that $u$ satisfies both properties. We show that $u$ is a unit of $\fa_{n,r}$. If $m \in \fM_n$ is of rank $\le r$ then we may write $m = g e_{n,r} g^{-1} \cdot B$, where $g \in \fG_n$ and $B \in \fM_n$. Then using the fact that $u = [g] u \left[g^{-1}\right]$, we have
\begin{displaymath}
u [m] = [g] u [g^{-1}] [g e_{n,r} g^{-1} B ] = [g] u [e_{n,r}] [g^{-1} B].
\end{displaymath}
Using the second property we arrive at $u \left[m\right] = \left[m\right]$, as required. 

We are left to show that $u$ is also a unit from the right of $\fa_{n,r}$. Consider the linear map $k\left[\fM_n\right] \to k\left[\fM_n\right]$ sending $\left[x\right]$ to $\left[\transpose{x}\right]$ for every $x \in \fM_n$, where $\transpose{x}$ is the transpose of $x$. Since $u$ is invariant under conjugation and since every element in $\fM_n$ is conjugate to its transpose, this linear map fixes $u$. By applying this linear map to the equality $u \left[m\right] = \left[m\right]$ we deduce $\left[\transpose{m}\right] u = \left[\transpose{m}\right]$ for any $m \in \fM_n$ with rank $\le r$. Thus we deduce that $u$ is also a unit from the right for $\fa_{r}$.
\end{proof}

We also require the following identity.

\begin{lemma} \label{lem:sum-identity}
Let $t \le r$ be positive integers. Then
\begin{displaymath}
\sum_{j=t}^r \mu(j) q^{-\left(r-1\right)j} \qbinom{r-t}{j-t} = \delta_{r,t} \cdot \mu(r)^{-1}.
\end{displaymath}
\end{lemma}

\begin{proof}
We have
\begin{displaymath}
\begin{split}
\sum_{j=t}^r \mu(j) q^{-\left(r-1\right)j} \qbinom{r-t}{j-t}
= & (-1)^t q^{\binom{t}{2}-\left(r-1\right)t} \sum_{j=0}^{r-t} (-1)^j q^{\binom{j}{2}} q^{-\left(r-t-1\right)j} \qbinom{r-t}{j} \\
= & (-1)^t q^{-\binom{t}{2}-t(r-t)} \prod_{j=0}^{r-t-1} \left(1 - q^{-(r-t-1)} q^j\right).
\end{split}
\end{displaymath}
The first step is an elementary manipulation, while the second uses the $q$-binomial theorem (\S \ref{ss:qbinom}). The final result clearly agrees with the stated formula.
\end{proof}

We are now ready to prove the theorem.

\begin{proof}[Proof of Theorem~\ref{thm:unit}]
It is clear that $\eta_r$ is invariant under conjugation by elements of $\fG_n$. Thus by Lemma \ref{lem:kuhn-criterion-for-unit}, it suffices to prove that $\eta_r \cdot [e_{n,r}] = [e_{n,r}]$. Using the expression for $\eta_r$ from Lemma~\ref{lem:eta-alt}, we find
\begin{displaymath}
\eta_r [e_{n,r}] = \sum_{j=0}^r q^{-\left(n-j\right)j} \mu(j)^{-1} \sum_{\substack{T_r \colon V_r \to V_n\\
		T_r \text{ semi-idempotent}}} \left[T_r\right] \sum_{\substack{S \colon V_n \to V_n\\
		S \text{ semi-idempotent}\\
		\rank S \le j \\
		S \restriction_{V_r} = T_r}} 
\mu(\srk{S}) \qbinom{n-\rank S}{j-\rank S},
\end{displaymath}
and so
\begin{displaymath}
\eta_r [e_{n,r}] = \sum_{j=0}^r q^{-\left(n-j\right)j} \mu(j)^{-1} \sum_{\substack{T_r \colon V_r \to V_n\\ T_r \text{ semi-idempotent}}} \left[T_r\right] h_0\left(j, T_r\right).
\end{displaymath}
By Proposition~\ref{prop:hk-independent-of-k}, $h_0(j,T_r)=h_{n-r}(j,T_r)$, and so
\begin{displaymath}
\eta_r [e_{n,r}] = \sum_{j=0}^r q^{-\left(r-j\right)j} \mu(j)^{-1} \sum_{\substack{T_r \colon V_r \to V_r\\
T_r \text{ semi-idempotent} \\
\rank T_r \le j}} \left[T_r\right] \mu(\srk{T_r}) \qbinom{r-\rank T_r}{j-\rank T_r}.
\end{displaymath}
Using the identity $\mu\left(j\right)^2 = q^{-\left(r-j\right)j} q^{\left(r-1\right)j}$, the last equality becomes 
\begin{displaymath}
\eta_r [e_{n,r}] = \sum_{\substack{T \colon V_r \to V_r\\ T \text{ semi-idempotent}}} \mu(\srk{T}) [T] \sum_{j=\rank T}^r \mu(j) q^{-\left(r-1\right)j} \qbinom{r-\rank T}{j-\rank T}.
\end{displaymath}
By Lemma~\ref{lem:sum-identity}, the inner most sum vanishes unless $\rank{T}=r$, in which case it equals $\mu(r)^{-1}$. Since the identity map $V_r \to V_r$ is the only semi-idempotent element $T \colon V_r \to V_r$ with rank $r$, and since it satisfies $\srk{T} = r$, the right side of the above equation is simply $[e_{n,r}]$, as required.
\end{proof}

\subsection{Proof of Theorem~\ref{thm:unit-decomp}} \label{ss:unit-decomp}

Denote $V = \bF^n$. We will identify elements of $\End_{\bF}\left(V\right)$ with $n \times n$ matrices $\rM_n\left(\bF\right)$ via the standard basis. Let $W \subset V$ be a linear subspace. For any linear subspace $U \subset V$ such that
\begin{displaymath}
V = W \oplus U
\end{displaymath}
we define an embedding
\begin{displaymath}
i_{W,U} \colon \End_{\bF}\left(W\right) \to \End_{\bF}\left(V\right)
\end{displaymath}
by setting $i_{W,U}\left(T\right) \restriction_W = T$ and $i_{W,U}\left(T\right) \restriction_U = 0$ for every linear map $T \colon W \to W$.

Define
\begin{displaymath}
\cE^{\ast}_{W,U} = -\mu(\dim W)^{-1} \sum_{\substack{T \colon W \to W\\ T \text{ semi-idempotent}\\
T \text{ singular}}} \mu(\srk{T}) \left[i_{W,U}\left(T\right)\right].
\end{displaymath}
It follows from Theorem~\ref{thm:unit} that $\cE^{\ast}_{W,U}$ is a unit for the monoid algebra spanned by $\left[i_{W,U}\left(T\right)\right]$ where $T$ goes over all the singular elements in $\End_{\bF}\left(W\right)$. In particular $\left(\cE^{\ast}_{W,U}\right)^2 = \cE^{\ast}_{W,U}$.

Next we define
\begin{displaymath}
\cE_{W,U} = \mu(\dim{W})^{-1} \sum_{\substack{T \colon W \to W\\
T \text{ semi-idempotent}}} \mu(\srk{T}) \left[i_{W,U}\left(T\right) \right].
\end{displaymath}
Notice that
\begin{displaymath}
\cE_{W,U} = \left[i_{W,U}\left(\id_{W}\right)\right] - \cE^{\ast}_{W,U}.
\end{displaymath}
The following statement follows from the fact that $\cE_{W,U}^{\ast}$ is an idempotent and that
\begin{displaymath}
	\left[i_{W,U}\left(\id_W\right)\right] \cE^{\ast}_{W,U} = \cE^{\ast}_{W,U} \left[i_{W,U}\left(\id_W\right)\right]  = \cE^{\ast}_{W,U}.
\end{displaymath}
\begin{proposition}\label{prop:e-idempotent-identity-V-equals-W}
Suppose that $V = W \oplus U$. Then
\begin{displaymath}
\cE_{W,U}^2 = \cE_{W,U}.
\end{displaymath}
\end{proposition}

Recall that we have defined
\begin{displaymath}
	\eta_W = q^{-\dim W \dim\left(V \slash W\right)} \mu(\dim{W})^{-1}
	\sum_{\substack{T \colon V \to V\\
			T \text{ semi-idempotent}\\
			\im T \subset W}} \mu(\srk{T}) \left[T\right].
\end{displaymath}
We require the following relations between $\cE_{W,U}$ and $\eta_W$:
\begin{proposition}\label{prop:eta-epsilon-identities}
	Suppose that $V = W \oplus U$. We have the following identities:
		\begin{displaymath}
			\eta_W \cdot \cE_{W,U} = \cE_{W,U}
		\end{displaymath}
	and
		\begin{displaymath}
			\cE_{W,U}  \cdot \eta_W = \eta_W.
		\end{displaymath}
\end{proposition}
\begin{proof}
	For linear maps $T \colon W \to W$ and $A \colon U \to W$ let $j_{W,U}\left(T,A\right) \colon V \to W$ be the linear map defined via the rule $j_{W,U}\left(T,A\right)\restriction_W = T$ and $j_{W,U}\left(T,A\right) \restriction_U = A$. Then \begin{displaymath}
		j_{W,U}\left(T_1,A_1\right) i_{W,U}\left(T_2\right) = i_{W,U}\left(T_1 T_2\right).
	\end{displaymath}
	It is clear that \begin{equation}\label{eq:expression-of-eta-W-with-j}
		\eta_W = q^{-\dim W \dim\left(V \slash W\right)} \mu(\dim{W})^{-1} \sum_{A \colon U \to W} \sum_{\substack{T \colon W \to W\\
				T \text{ semi-idempotent}\\
				\im T \subset W}} \mu\left(\srk T\right) j_{W,U}\left(T, A\right).
	\end{equation}
	Thus $\eta_W \left[\id_{W,U}\left(\id_W\right)\right]$ equals
	\begin{displaymath}
		q^{-\dim W \dim\left(V \slash W\right)} \mu(\dim{W})^{-1} q^{\dim U \dim W} \sum_{\substack{T \colon W \to W\\
				T \text{ semi-idempotent}\\
				\im T \subset W}} \mu\left(\srk T\right) i_{W,U}\left(T\right) = \cE_{W,U}.
	\end{displaymath}
	This implies
	\begin{displaymath}
		\eta_W \cE_{W,U} = \eta_W \left[i_{W,U}\left(\id_{W}\right)\right] \cE_{W,U} = \cE_{W,U}^2 = \cE_{W,U},
	\end{displaymath}
	which proves the first identity.
	
	Regarding the second identity, this identity is equivalent to the identity
	\begin{displaymath}
		\cE^{\ast}_{W,U} \cdot \eta_W = 0.
	\end{displaymath}
	Since $\cE_{W,U}^{\ast}$ is the identity of the subalgebra spanned by $\left[i_{W,U}\left(T\right)\right]$ where $T$ runs over all singular elements $T \colon W \to W$, the last identity is equivalent to the identity
	\begin{displaymath}
		\left[i_{W,U}\left(T\right)\right] \cdot \eta_W = 0
	\end{displaymath}
	for any singular map $T \colon W \to W$. Choose a basis $b_1,\dots,b_m$ for $W$ and a basis $b_{m+1},\dots,b_{n}$ for $U$ and identify linear maps $T \colon W \to W$ with $m \times m$ matrices and linear maps $V \to V$ with $n \times n$ matrices. We claim that it suffices to prove that \begin{displaymath}
		\left[i_{W,U}\left(e_{m,m-1}\right)\right] \eta_W = 0.
	\end{displaymath}
	Indeed, by \eqref{eq:expression-of-eta-W-with-j} we have that for any $g \in \GL_m\left(\bF\right)$ \begin{displaymath}
		\left[\diag\left(g, I_{n-m}\right)\right] \left[i_{W,U}\left(e_{m,m-1}\right)\right] \eta_W \left[\diag\left(g^{-1}, I_{n-m}\right)\right] = \left[i_{W,U}\left(g e_{m,m-1} g^{-1}\right)\right] \eta_W.
	\end{displaymath}
	As in Lemma~\ref{lem:kuhn-criterion-for-unit}, we can write every singular linear map $T \colon W \to W$ as $T = B g e_{m,m-1} g^{-1}$, where $B \colon W \to W$ is a linear map. Thus if $\left[i_{W,U}\left(e_{m,m-1}\right)\right] \eta_W = 0$ then
	\begin{displaymath}
		\left[i_{W,U}\left(T\right)\right] \eta_W = [\diag\left(B, I_{n-m}\right)] [\diag\left(g, I_{n-m}\right)] \left[i_{W,U}\left(e_{m,m-1}\right)\right] \eta_W [\diag\left(g, I_{n-m}\right)]^{-1} = 0.
	\end{displaymath}
	
	To prove that $\left[i_{W,U}\left(e_{m,m-1}\right)\right] \eta_W = 0$ we will use Lemma~\ref{lem:coordinate-free-extension-of-corank-1-partial-semi-idempotent}. Using \eqref{eq:expression-of-eta-W-with-j} again we have that for matrices with last row zero $A$ and $T$ of sizes $m \times \left(n-m\right)$ and $m \times m$, respectively, the coefficient of $\left[j_{W,U}\left(T,A\right)\right]$ in $\left[i_{W,U}\left(e_{m,m-1}\right)\right] \eta_W$ equals
	\begin{displaymath}
		q^{-\dim W \dim\left(V \slash W\right)} \mu(\dim{W})^{-1} \cdot q^{n-m} \cdot \sum_{\substack{S \colon W \to W\\
				S \text{ semi-idempotent}\\
				e_{m,m-1} S = T}} \mu\left(\srk S\right).
	\end{displaymath}
	Here the factor $q^{n-m}$ is the number of matrices that can be obtained by changing the last row of $A$. Consider the inner sum
	\begin{displaymath}
		\sum_{\substack{S \colon W \to W\\
				S \text{ semi-idempotent}\\
				e_{m,m-1} S = T}} \mu\left(\srk S\right).
	\end{displaymath}
	Applying the transpose map to the condition $e_{m,m-1} S = T$ gives a condition on the restriction of the semi-idempotent element $\transpose{S}$ to the codimensional $1$ subspace $V_{m-1} \subset W$ defined by 
	\begin{displaymath}
		V_{m-1} = \Span_{\bF}\left(b_1,\dots,b_{m-1}\right).
	\end{displaymath}
	In particular, $\transpose{T} \restriction_{V_{m-1}}$ has to be semi-idempotent. Using the notation of Lemma~\ref{lem:coordinate-free-extension-of-corank-1-partial-semi-idempotent}, we have that the last sum equals
	\begin{displaymath}
		\mu\left(\srk T\right) e\left(\transpose{T} \restriction_{V_{m-1}}, \srk T\right) + \mu\left(\srk T + 1\right) e\left(\transpose{T} \restriction_{V_{m-1}}, \srk T+1\right),
	\end{displaymath}
	which by Lemma~\ref{lem:coordinate-free-extension-of-corank-1-partial-semi-idempotent} equals zero, as $\mu\left(\srk T + 1\right) = -\mu\left(\srk T\right) q^{\srk T}$ and $\srk T = \srk \transpose{T}$. Thus $\left[i_{W,U}\left(e_{m,m-1}\right)\right]\eta_W =0$ and the second identity is proved.
\end{proof}
Next we show that $\eta_{W}$ are orthogonal for different $W$.
\begin{proposition}\label{prop:orthogonality-of-eta}
	Let $W_1, W_2 \subset V$. If $W_1 \ne W_2$ then $\eta_{W_1} \eta_{W_2} = 0$.
\end{proposition}
\begin{proof}
	Choose decompositions $V = W_1 \oplus U_1 = W_2 \oplus U_2$ such that $U_2 \cap W_1 \ne 0$. Then \begin{displaymath}
		\eta_{W_1} \eta_{W_2} = \eta_{W_1} \left[i_{W_2,U_2}\left(\id_{W_2}\right)\right] \eta_{W_2}.
	\end{displaymath} 
	We will show that
	\begin{displaymath}
		\eta_{W_1} \left[i_{W_2, U_2}\left(\id_{W_2}\right)\right] = 0.
	\end{displaymath}
	Let $b_1,\dots,b_m$ be a basis of $W_1$ such that $b_m \in U_2$. Let $b_{m+1}, \dots, b_n$ be a basis of $U_1$. Identifying linear maps $T \colon V \to V$ with $n \times n$ matrices we notice that
	\begin{displaymath}
		\diag\left(e_{m,m-1}, I_{n-m} \right) \cdot i_{W_2,U_2}\left(\id_{W_2}\right) = i_{W_2,U_2}\left(\id_{W_2}\right)
	\end{displaymath}
	and thus it suffices to show that $\eta_{W_1} \left[\diag\left(e_{m,m-1}, I_{n-m}\right)\right] = 0$.
	
	By \eqref{eq:expression-of-eta-W-with-j}, $q^{\dim W_1 \dim\left(V \slash W_1\right)} \mu(\dim{W_1}) \eta_{W_1} \left[\diag\left(e_{m,m-1}, I_{n-m}\right)\right]$ equals
	\begin{displaymath}
		\sum_{A \colon U_1 \to W_1} \sum_{\substack{T \colon W_1 \to W_1\\
				T \text{ semi-idempotent}}} \mu\left(\srk T\right) \left[j_{W,U}\left(T e_{m,m-1}, A\right)\right].
	\end{displaymath}
	Let $V_{m-1}$ be the subspace of $W_1$ of co-dimension $1$ defined by $V_{m-1} = \Span_{\bF}\left(b_1,\dots,b_{m-1}\right)$. Let $S \colon W_1 \to W_1$ be a linear map such that $S b_m = 0$ and such that $S\restriction_{V_{m-1}}$ is semi-idempotent, and let $A \colon U_1 \to W_1$ be a linear map. Using the notation of Lemma~\ref{lem:coordinate-free-extension-of-corank-1-partial-semi-idempotent}, we have that the coefficient of $\left[j_{W,U}\left(S, A\right)\right]$ in the last sum is
	\begin{displaymath}
		\mu\left(\srk S\right) e\left(S \restriction_{V_{m-1}}, \srk\right) + \mu\left(\srk S + 1\right) e\left(S \restriction_{V_{m-1}}, \srk + 1\right).
	\end{displaymath}
	As in the previous proof, Lemma~\ref{lem:coordinate-free-extension-of-corank-1-partial-semi-idempotent} implies that this expression is zero.
\end{proof}

We are now ready to prove the theorem.

\begin{proof}[Proof of Theorem~\ref{thm:unit-decomp}]
We break the proof into several steps.

\textit{Step 1: idempotence.}
By Proposition \ref{prop:eta-epsilon-identities} we have that
\begin{displaymath}
\eta_W^2 = \cE_{W,U} \left(\eta_W \cE_{W,U}\right) \eta_W = \cE_{W,U} \cE_{W,U} \eta_W = \cE_{W,U} \eta_W = \eta_W,
\end{displaymath}
and so $\eta_W$ is idempotent.

\textit{Step 2: orthogonality.}
This follows from Proposition~\ref{prop:orthogonality-of-eta}.

\textit{Step 3: expression for $\eta_r$.}
From Lemma~\ref{lem:eta-alt}, we have
\begin{equation*}
\eta_r = \sum_{j=0}^r q^{-\left(n-j\right)j} \mu(j)^{-1} \sum_{\substack{T \in \End_{\bF}\left(V\right)\\
\rank T \le j\\
T \text{ semi-idempotent}}} \mu(\srk{T}) \qbinom{n - \rank T }{j - \rank T} \left[T\right].
\end{equation*}
Notice that for $T$ in the inner sum, we have that $\im T$ is a subspace of $V$ of dim $\le j$ and that the $q$-binomial $\qbinom{r-\rank T}{j-\rank T}$ counts the number of subspaces $W \subset V$ of dimension $j$ such that $\im T \subset W$. Thus we may rewrite $\eta_r$ as
\begin{displaymath}
\eta_r = \sum_{j=0}^r \sum_{\substack{W \subset V\\ \dim W = j}} q^{-\left(n-j\right)j} \mu(j)^{-1} \sum_{\substack{T \in \End_{\bF}\left(V\right)\\
\im T \subset W \\
T \text{ semi-idempotent}}} \mu(\srk{T}) \left[T\right].
\end{displaymath}
From this we obtain the expression for $\eta_r$.

\textit{Step 4: computation of $\psi_r$.}
We now explain the formula for $\psi_r(\eta_W)$. It is clear that $\psi_r(\eta_W)=0$ if $\dim{W}<r$. Suppose now that $\dim{W}=r$. Let $W' \subset V$ be another subspace of dimension $r$. Choose decompositions $V = W \oplus U = W' \oplus U'$ such that $U' \cap W \ne 0$. As in the proof of Proposition~\ref{prop:orthogonality-of-eta}, we have that $\eta_{W} \left[i_{W', U'}\left(\id_{W'}\right)\right] = 0$. Thus $W$ is the only space that can contribute to $\psi_r\left(\eta_W\right)$. Since $\eta_W \left[i_{W,U}\left(\id_{W}\right)\right] = \cE_{W,U}$, we see that $\psi_r\left(\eta_W\right) = \psi_r\left(\cE_{W,U}\right)$. It is clear that $\psi_r(\cE_{W,U}^*)=0$, and so $\psi_r(\cE_{W,U})=[\id_W]$.
\end{proof}

\begin{remark} \label{rmk:solomon}
In \cite{Solomon3}, Solomon gives a similar formula for the unit of the rook monoid algebra. Our notation differs from his. His $\mu_r$ is analogous to our $\eta_r - \eta_{r-1}$. In addition, his $\cE_{K}$ is analogous to our 
\begin{displaymath}
\cE'_{W} = \sum_{W' \subset W} \eta_{W'},
\end{displaymath}
for a linear subspace $W \subset \bF^n$. Using the orthogonality relations proved above it is easy to check that $\cE'_{W_1} \cdot \cE'_{W_2} = \cE'_{W_1 \cap W_2}$. It can be shown that
\begin{displaymath}
\cE'_{W} = \mu\left(\dim W\right)^{-1} \sum_{\substack{T \colon V \to V\\
	T \text{ semi-idempotent} \\
	\im T = W}} \mu(\srk{T}) \left[T\right].
\end{displaymath}
Thus
\begin{displaymath}
	\eta_{W} = q^{-\dim W \left(n-\dim W\right)}  \mu\left(\dim W\right)^{-1} \sum_{W' \subset W} \mu\left(\dim W'\right) \cE'_{W'}.
\end{displaymath}
\end{remark}

\section{Module theory}\label{s:module-theory}

We now investigate the structure of $k[\fM_n]$-modules. We assume throughout that $k$ is algebraically closed of characteristic~0. Weaker hypotheses suffice in parts, but to simplify exposition we make this blanket assumption.

\subsection{Simple modules} \label{ss:simple}

We have seen (Proposition~\ref{prop:structure}) that $k[\fM_n]$ is isomorphic to a product of groupoid algebras $k[\cG_{n,r}]$; moreover, $k[\cG_{n,r}]$ is isomorphic to a matrix algebra over $k[\fG_r]$ (Proposition~\ref{prop:groupoid}). It follows that simple $k[\fM_n]$-modules correspond to simple $k[\fG_r]$-modules, for $0 \le r \le n$. For a simple $k[\fG_r]$-module $\pi$, we let $L_n(\pi)$ be the corresponding simple $k[\fM_n]$-module.

We now aim to describe $L_n(\pi)$ explicitly. For this, we recall a result from Morita theory. Let $A$ be a finite dimensional $k$-algebra, and fix $n \ge 1$. Then the category of $A$-modules is equivalent to the category of $\rM_n(A)$-modules. If $V$ is an $A$-module the associated $\rM_n(A)$-module is $V \otimes k^n$. The action is the natural one if we identify $M_n(A)$ with $A \otimes_k M_n(k)$. See \cite[Example~A.27]{Steinberg}.

We now describe $L_n(\pi)$. Thus fix an irreducible representation $\pi \colon \fG_r \to \GL(V)$. Let $U_1, \ldots, U_N$ be the $r$-dimensional subspaces of $\bF^n$, and for each $1 \le i \le N$, fix an isomorphism $f_i \colon \bF^r \to U_i$. For each $r$-dimensional subspace $U$ of $\bF^n$, fix an isomorphism $f_i \colon \bF^r \to U$. By Proposition~\ref{prop:groupoid}, we have an isomorphism
\begin{displaymath}
k[\cG_{n,r}] \to \rM_N(k[\fG_r]), \qquad [a \colon U_i \to U_j] \mapsto [f_j^{-1} a f_i] E_{j,i},
\end{displaymath}
where $E_{j,i} \in \rM_N(k)$ is the usual elementary matrix. We thus obtain a left $k[\cG_{n,r}]$-module structure on $V \otimes k^N$ by letting $[a]$ act by $\pi(f_j^{-1} a f_i) \otimes E_{j,i}$, where $a \colon U_i \to U_j$. We obtain a left $k[\fM_n]$-module structure on $V \otimes k^N$ by composing with $\psi_r$; this is the module $L_n(\pi)$. For $m \in \fM_n$, the action of $[m]$ is by the operator $\sum_{(i,j) \in S} \pi(f_j^{-1} m f_i) \otimes E_{j,i}$ where $S \subset [N] \times [N]$ consists of those pairs $(i,j)$ such that $m(U_i)=U_j$. Here $[N]=\{1, \ldots, N\}$.

As a corollary of the above discussion, we obtain a formula for the character of $L_n(\pi)$.

\begin{proposition}
Using the above notation, we have
\begin{displaymath}
\tr(m \vert L_n(\pi)) = \sum_{(i,i) \in S} \tr(\pi(f_i^{-1} m f_i)).
\end{displaymath}
\end{proposition}

\begin{remark}
This formula is very similar to the analogous one for the rook monoid obtained by Munn, see \cite[Theorem~2.30]{Solomon3}.
\end{remark}

\subsection{Restriction to the group} \label{ss:res-group}

Since $\fG_n$ is a submonoid of $\fM_n$, any $k[\fM_n]$-module can be restricted to $k[\fG_n]$. We now examine how simple modules restrict.

\begin{proposition} \label{prop:res-group}
Let $\pi$ be a simple $k[\fG_r]$-module. Then the restriction of $L_n(\pi)$ to $k[\fG_n]$ is isomorphic to the parabolic induction $\Ind_{\fP_{r,n-r}}^{\fG_n}(\pi \boxtimes \bone)$.
\end{proposition}

\begin{proof}
Use notation as in \S \ref{ss:simple}, and choose $U_1$ to be the standard $r$-dimensional subspace $\bF^r$ of $\bF^n$, so that its stabilizer in $\fG_n$ is the parabolic subgroup $\fP_{r,n-r}$. Let $u_1, \ldots, u_N$ be the basis of $k^N$, where $u_i$ corresponds to $U_i$. We have a decomposition
\begin{displaymath}
V \otimes k^N = \bigoplus_{i=1}^N V \otimes ku_i.
\end{displaymath}
Each $g \in \fG_n$ induces a permutation of the set $\{1,\ldots,N\}$, through its action on $r$-dimensional subspaces, and we have $g(V \otimes ku_i) = V \otimes ku_{gi}$. It follows that the space $V \otimes ku_1$ is stable under $\fP_{r,n-r}$. Moreover, the action of $g \in \fP_{r,n-r}$ on this space is by the operator $\pi(\ol{g})$, where $\ol{g}$ is the image of $g$ under the quotient map $\fP_{r,n-r} \to \fG_r$. These claims follow from the explicit description of the module structure given in \S \ref{ss:simple}, and they imply the present proposition.
\end{proof}

\begin{remark}
We can make the above proposition more explicit. We have a map
\begin{displaymath}
\phi_0 \colon V \to V \otimes k^N, \qquad x \mapsto x \otimes u_1
\end{displaymath}
of $k[\fP_{r,n-r}]$-modules. The target here is $L_n(\pi)$, which is a $k[\fG_n]$-module. Thus, by Frobenius reciprocity, we have a map
\begin{displaymath}
\phi \colon k[\fG_n] \otimes_{k[\fP_{r,n-r}]} \pi \to V \otimes k^N
\end{displaymath}
of $k[\fG_n]$-modules. The explicit description of this map uses the definition of the $\fG_n$-action on the target given in \S \ref{ss:simple}. Let $g \in \fG_n$ and $x \in V$ be given, and $gU_1=U_j$. Then
\begin{displaymath}
\phi(g \otimes x) = \pi(f_j^{-1}gf_1)x  \otimes u_j
\end{displaymath}
where the $f_i$'s are as in \S \ref{ss:simple}. This $\phi$ is the isomorphism constructed in Proposition~\ref{prop:res-group}.
\end{remark}

Combining the above proposition with the theory from \S \ref{ss:repGLn} allows us to determine the irreducible decomposition of $L_n(\pi)$ as a $k[\fG_n]$-module. 

\begin{corollary} \label{cor:RestrictionCombinatorics}
Let $\pi$ be a simple $k[\fG_r]$-module, and write $\pi=\pi_0\{\lambda\}$ where $\pi_0$ is a pure $k[\fG_s]$-module and $\lambda$ is a partition of $r-s$. Then the restriction of $L_n(\pi)$ to $k[\fG_n]$ decomposes as $\bigoplus_{\mu} \pi_0\{\mu\}$, where the sum is over partitions $\mu$ of $n-s$ such that $\mu/\lambda \in \HS$.
\end{corollary}

\subsection{Restriction to smaller monoids} \label{ss:res-monoid}

Suppose $A$ is a ring and $e \in A$ is an idempotent. We then have an exact restriction functor
\begin{displaymath}
\Mod_A \to \Mod_{eAe}, \qquad M \mapsto eM,
\end{displaymath}
and an induction functor
\begin{displaymath}
\Mod_{eAe} \to \Mod_A, \qquad N \mapsto Ae \otimes_{eAe} N.
\end{displaymath}
The general formalism of these functors is studied in \cite[\S 4]{Steinberg}.

We apply this with $A=k[\fM_n]$ and $e=[e_{n,s}]$. In this case, we have a natural identification of $eAe$ with $k[\fM_s]$. We thus have restriction and induction functors
\begin{displaymath}
\Res_{n,s} \colon \Mod_{k[\fM_n]} \to \Mod_{k[\fM_s]}, \qquad
\Ind_{n,s} \colon \Mod_{k[\fM_s]} \to \Mod_{k[\fM_n]}.
\end{displaymath}
We now examine how they behave on simple modules.

\begin{proposition} \label{prop:res-monoid}
Let $\pi$ be a simple $k[\fG_r]$-module. Then $\Res_{n,s}(L_n(\pi))$ is isomorphic to $L_s(\pi)$ if $s \ge r$, and vanishes if $s<r$.
\end{proposition}

\begin{proof}
Use notation as in \S \ref{ss:simple}, and let $M=V \otimes k^N$ be the simple $L_n(\pi)$. If follows immediately from the discussion in \S \ref{ss:simple} that $\fa_{n,r-1}$ annihilates $M$. Thus if $s<r$ then $eM=0$, and the result follows. Assume now that $s \ge r$. Identify $\bF^s$ with a subspace of $\bF^n$ in the usual manner, and label the subspaces of $\bF^n$ so that $U_1, \ldots, U_{N'}$ are exactly the $r$-dimensional subspaces contained in $\bF^s$. It follows from the discussion in \S \ref{ss:simple} that $eM=V \otimes k^{N'}$, which is exactly the space $L_s(\pi)$. The discussion in \S \ref{ss:simple} shows that the action of $k[\fM_s]$ on $eM$ agrees with the action on $L_s(\pi)$. Thus the result follows.
\end{proof}

\begin{corollary}
Let $\pi$ be a a simple $k[\fG_r]$-module, with $r \le s$. Then $\Ind_{n,s}(L_s(\pi))=L_n(\pi)$.
\end{corollary}

\begin{proof}
Since induction is the left adjoint of restriction, it follows from the proposition that $L_n(\pi)$ is a summand of $\Ind_{n,s}(L_s(\pi))$. On the other hand, the induction is indecomposable \cite[Corollary~4.10]{Steinberg}.
\end{proof}

\subsection{Grothendieck groups} \label{ss:groth}

Let $\rK(k[\fM_n])$ be the Grothendieck group of the category of finite dimensional $k[\fM_n]$-modules. This is the free $\bZ$-module on the classes $[L_n(\pi)]$, where $\pi$ is a simple $k[\fG_r]$-module with $0 \le r \le n$. For $0 \le r \le n$, we have maps
\begin{displaymath}
\rK(k[\fM_n]) \to \rK(k[\fM_r]) \to \rK(k[\fG_r]),
\end{displaymath}
where the first map is the restriction map discussed in \S \ref{ss:res-monoid}, and the second is the restriction map discussed in \S \ref{ss:res-group}. The resulting map
\begin{displaymath}
\rK(k[\fM_n]) \to \prod_{r=0}^n \rK(k[\fG_r]).
\end{displaymath}
is an isomorphism. This is a general result \cite[Theorem~6.5]{Steinberg} from the representation theory of monoids, but Corollary~\ref{cor:RestrictionCombinatorics} and Proposition~\ref{prop:res-monoid} give an explicit description of what this map is in terms of simple modules.

\subsection{Duality} \label{ss:dual}

The monoid $\fM_n$ has an anti-involution given by transpose, which induces an anti-involution on the algebra $k[\fM_n]$. Using this, we can convert right modules to left modules. In particular, if $M$ is a left $k[\fM_n]$-module then the $k$-linear dual $M^*$ is naturally a left $k[\fM_n]$-module. We now determine the effect of this operation on simple modules.

\begin{proposition} \label{prop:dual}
The simple module $L_n(\pi)$ is self-dual.
\end{proposition}

\begin{proof}
Suppose $V$ is a representation of the group $\fG_n$. The dual space $V^*$ is then a representation of $\fG_n$ via the formula $(g \lambda)(v)=\lambda(g^{-1} v)$, where here $\lambda \in V^*$, $v \in V$, and $g \in \fG_n$. For the purposes of this proof, we will use a different representation of $\fG_n$ on $V^*$, defined by $(g \lambda)(v)=\lambda(g^t v)$. For any $g \in \fG_n$, the elements $g$ and $g^t$ are conjugate. It follows that $V$ and $V^*$ (with the transpose action) have the same character, and thus are isomorphic representations. Thus duality (defined with transpose) acts trivially on $\rK(\fG_n)$.

The restriction functor $\Mod_{k[\fM_n]} \to \Mod_{k[\fM_r]}$ is compatible with duality, since the idempotent $e_{n,r}$ is its own transpose. The restriction functor $\Mod_{k[\fM_n]} \to \Mod_{k[\fG_n]}$ is also compatible with duality, provided we use the transpose form discussed above. We thus see that the map
\begin{displaymath}
\rK(k[\fM_n]) \to \prod_{r=0}^n \rK(k[\fG_n])
\end{displaymath}
is compatible with duality. Since this is an isomorphism (\S \ref{ss:groth}) and duality is trivial on the target, it follows that duality is trivial on $\rK(k[\fM_n])$, which completes the proof.
\end{proof}

\begin{remark}
It would be interesting to precisely understand how the homomorphism $\psi_r$ from \S \ref{ss:kovacs} interacts with transpose. Similarly, it would be interesting to understand the operation on $k[\cG_{n,r}]$ induced by transpose.
\end{remark}

\subsection{The multiplicity formula}

Let $M$ be a finite dimension $k[\fM_n]$-module. We have a decomposition
\begin{displaymath}
M = \bigoplus_{\pi} L_n(\pi)^{\oplus \bm(\pi)},
\end{displaymath}
and we would like to determine the multiplicities $\bm(\pi)$. Let $M_r=[e_{n,r}] M_n$ be the restriction of $M$ to $k[\fM_r]$, as discussed in \S \ref{ss:res-monoid}. For an irreducible $k[\fG_r]$-module $\pi$, let $\bn(\pi)$ be the multiplicity of $\pi$ in $M_r$, as a $k[\fG_r]$-module. This can, in principle, be determined using the character theory of $\fG_r$. The following theorem allows one to determine $\bm(\pi)$ from the $\bn(\pi)$'s. We let $\VS$ be the set of Young diagrams that are vertical strips, and define $\lambda/\mu \in \VS$ analogously to $\lambda/\mu \in \HS$.

\begin{theorem} \label{thm:multiplicity}
Let $\pi$ be a pure representation of $\fG_r$, and let $\lambda$ be a partition with $r+\vert \lambda \vert \le n$.
\begin{enumerate}
\item We have
\begin{displaymath}
\bn(\pi\{\lambda\})=\sum_{\lambda/\mu \in \HS} \bm(\pi\{\mu\}).
\end{displaymath}
\item We have
\begin{displaymath}
\bm(\pi\{\lambda\})=\sum_{\lambda/\mu \in \VS}(-1)^{\vert \lambda \vert-\vert \mu \vert} \bn(\pi\{\mu\}).
\end{displaymath}
\end{enumerate}
\end{theorem}

\begin{proof}
(a) It suffices to prove the formula when $M$ is a simple module. Thus suppose $M=L_n(\sigma)$, where $\sigma$ is a simple $k[\fG_t]$-module, and write $\sigma=\sigma_0\{\rho\}$ for some pure representation $\sigma_0$ and partition $\rho$. By Proposition~\ref{prop:res-monoid}, we have $M_s=L_s(\sigma)$ if $s \ge t$, and $M_s=0$ otherwise. If $\sigma_0 \ne \pi$ then both sides vanish. Indeed, this is clear for the right side, and follows from Corollary~\ref{cor:RestrictionCombinatorics} (applied to the $M_s$) for the left side.

Suppose now that $\sigma_0=\pi$. The right side is~1 if $\lambda/\rho \in \HS$, and vanishes otherwise. Let $s=r+\vert \lambda \vert$. If $\vert \lambda \vert<\vert \rho \vert$ then $M_s=0$, and so both sides clearly vanish. Suppose then that $\vert \lambda \vert \ge \vert \rho \vert$, so that $M_s=L_s(\pi\{\rho\})$. By Corollary~\ref{cor:RestrictionCombinatorics}, $M_s$ decomposes as a $k[\fG_s]$-module into $\bigoplus \pi\{\nu\}$ where $\nu/\rho \in \HS$ and $\vert \nu \vert=\vert \lambda \vert$. We thus see that the left side is~1 if $\lambda/\rho \in \HS$, and vanishes otherwise. This proves the formula.

(b) This follows from (a) and the subsequent lemma.
\end{proof}

Let $\cP$ be the set of all partitions, and let $V$ be the space of $k$-valued functions on $\cP$. Given $\phi \in V$, define new functions $A(\phi)$ and $B(\phi)$ in $V$ by
\begin{displaymath}
(A\phi)(\lambda) = \sum_{\lambda/\mu \in \HS} \phi(\mu), \qquad
(B\phi)(\lambda) = \sum_{\lambda/\mu \in \VS} (-1)^{\vert \lambda \vert-\vert \mu \vert} \phi(\mu).
\end{displaymath}
Thus $A$ and $B$ are linear operators on the space $V$.

\begin{lemma}
The operators $A$ and $B$ are inverse to one another.
\end{lemma}

\begin{proof}
Let $\Lambda$ be the ring of symmetric functions with coefficients in $k$. This is a $k$-vector space with a basis given by Schur functions $s_{\lambda}$, for $\lambda \in \cP$. Let $\hat{\Lambda}$ be the completion of $\Lambda$ with respect to the ideal of positive degree elements; thus an element of $\hat{\Lambda}$ is a formal $k$-linear combination of $s_{\lambda}$'s. We have an isomorphism $F \colon V \to \hat{\Lambda}$ by $F(\phi)=\sum_{\lambda} \phi(\lambda) s_{\lambda}$.

Let $h_n$ and $e_n$ be the usual complete and elementary symmetric functions. Put $h=\sum_{n \ge 0} h_n$ and $e=\sum_{n \ge 0} (-1)^n e_n$. We now examine operators on $\hat{\Lambda}$ given by multiplication by $h$ and $e$. By the Pieri rule, we have
\begin{displaymath}
h \cdot s_{\lambda} = \sum_{\mu/\lambda \in \HS} s_{\mu}, \qquad
e \cdot s_{\lambda} = \sum_{\mu/\lambda \in \VS} (-1)^{\vert \mu \vert - \vert \lambda \vert} s_{\mu}
\end{displaymath}
and so
\begin{displaymath}
h \cdot \sum_{\lambda} a_{\lambda} s_{\lambda} = \sum_{\mu} \big( \sum_{\mu/\lambda \in \HS} a_{\lambda} \big) s_{\mu}, \qquad
e \cdot \sum_{\lambda} a_{\lambda} s_{\lambda} = \sum_{\mu} \big( \sum_{\mu/\lambda \in \VS} (-1)^{\vert \mu \vert - \vert \lambda \vert}  a_{\lambda} \big) s_{\mu}.
\end{displaymath}
This shows that $F(A(\phi))=h F(\phi)$ and $F(B(\phi))=e F(\phi)$. The result now follows from the standard identity $eh=1$.
\end{proof}

\section{Schur--Weyl duality}\label{s:schur-weyl}

In this section, we investigate an analog of Schur--Weyl duality for $k[\fM_n]$. The reader should consult \cite[\S 5]{Solomon3} for an similar analog for the rook monoid.

\subsection{The main result}

The monoid $\fM_n$ acts on the set $\bF^n$, and so $V=k[\bF^n]$ is naturally a $k[\fM_n]$-module. We now formulate a Schur--Weyl duality statement for matrix monoids acting on tensor powers $V^{\otimes m}$. 
 
First observe that we have natural identifications
\begin{displaymath}
V^{\otimes m} = k[\bF^n]^{\otimes m} = k[\bF^n \times \bF^n \times \cdots \times \bF^n] = k[\fM_{n,m}]
\end{displaymath}
where there are $m$ factors in the product, and $\fM_{n,m}$ is the set of $n \times m$ matrices over $\bF$. The space $k[\fM_{n,m}]$ carries a natural right module structure for the algebra $k[\fM_m]$, which commutes with the left action of $k[\fM_n]$. The following is our Schur--Weyl theorem:

\begin{theorem} \label{thm:doublecentralizer}
The images of $k[\fM_n]$ and $k[\fM_m]$ in $\End(k[\fM_{n,m}])$ are centralizers of each other.
\end{theorem}

\begin{proof}
Since $k[\fM_n]$ and $k[\fM_{m}]$ are semi-simple, it suffices just to show $k[\fM_m]$ generates the centralizer of $k[\fM_n]$ on $k[\fM_{n,m}]$. Moreover, taking the transposes if necessary, we may assume without loss of generality that $n \ge m$.

A $k[\fM_n]$-equivariant endomorphism of  $k[\fM_{n,m}]$  is completely determined by what it does to the matrix
\begin{displaymath}
S = 
\begin{pmatrix}
I_{m} \\
 0_{(n-m) \times m}
\end{pmatrix}
\end{displaymath}
since any other matrix is in the $\fM_n$-orbit of this matrix. This matrix $S$ is fixed by 
\begin{displaymath}
T=e_{m,n}=
\begin{pmatrix}
I_{m} & \\
 & 0_{n-m}
\end{pmatrix} \in  \fM_n,
\end{displaymath}
so under an equivariant map it can only be sent to something else fixed by $T$.  The matrix $T$ acts on $k[\fM_{n,m}]$ as projection onto the space of linear combinations of matrices of the form
\begin{displaymath}
\begin{pmatrix}
A \\ 0_{(n-m) \times m}
\end{pmatrix}
\end{displaymath}
so we see that $S$ must get sent to a linear combination of matrices of this form.  Clearly though the matrix above can be obtained from the starting matrix $S$ by multiplying on the right by $A \in \fM_{m}$. So we see that indeed any $k[\fM_{n}]$-equivariant endomorphism can be realized by a linear combination of elements of $\fM_{m}$, as desired.
\end{proof}

\subsection{Decomposition into simples}

We have been treating $k[\fM_{n,m}]$ as a $(k[\fM_n], k[\fM_m])$-bimodule. Composing with transpose on the second algebra (as in \S \ref{ss:dual}), we now treat it as a left module for $k[\fM_n] \otimes k[\fM_m]$. As is typical for Schur--Weyl results, we can decompose $k[\fM_{n,m}]$ into simple modules.

\begin{proposition} \label{prop:SWdecomposition}
We have a decomposition of left $(k[\fM_{n}] \otimes k[\fM_{m}])$-modules
\begin{displaymath}
k[\fM_{n,m}] \cong \bigoplus L_n(\pi) \otimes L_m(\pi)
\end{displaymath}
where $\pi$ runs over all irreducible $\fG_r$-representations for $r \le \min(n,m)$. 
\end{proposition}

\begin{proof}
For any semi-simple $k$-algebra $A$, we have a natural isomorphism of $(A,A)$-bimodules
\begin{displaymath}
A = \bigoplus L \otimes L^*,
\end{displaymath}
where the sum is over simple left $A$-modules $L$, and $L^*$ is regarded as a right $A$-module. Applying this with $A=k[\fM_n]$, converting right modules to left modules via transpose, and using Proposition~\ref{prop:dual}, we obtain an isomorphism of left $(k[\fM_n] \otimes k[\fM_n])$-modules
\begin{equation} \label{eq:SWdecomposition}
k[\fM_n] \cong \bigoplus L_n(\pi) \otimes L_n(\pi),
\end{equation}
where $\pi$ runs over all irreducible $\fG_r$-representations for $r \le n$. Recall the idempotent $[e_{n,m}] \in k[\fM_n]$ from \S \ref{ss:res-monoid}. We now apply the idempotent $1 \otimes [e_{n,m}]$ to \eqref{eq:SWdecomposition}. A direct computation shows that
\begin{displaymath}
(1 \otimes [e_{n,m}]) k[\fM_n] \cong k[\fM_{n,m}]
\end{displaymath}
as left $(k[\fM_n] \otimes k[\fM_m])$-modules. Proposition~\ref{prop:res-monoid} computes the result of applying the idempotent to the right side of \eqref{eq:SWdecomposition}. This completes the proof.
\end{proof}

\begin{remark}
	Proposition~\ref{prop:SWdecomposition} is closely related to the theta correspondence for type II. The theta correspondence for type II concerns the $\GL_n\left(F\right) \times \GL_m\left(F\right)$ action on $L^2\left(M_{n \times m}\left(F\right)\right)$ given by $\rho\left(g,h\right)f\left(X\right) = f\left(\transpose{g} X h\right)$. Here $F$ is usually taken to be a local field. One can replace $\transpose{g}$ with $g^{-1}$ by composing with the outer automorphism of $\GL_n\left(\bF\right)$ given by $g \mapsto \transpose{\left(g^{-1}\right)}$. For an irreducible representation $\tau$ of $\GL_n\left(F\right)$ one considers the big theta space 
	\begin{displaymath}
		\Theta\left(\tau\right) = L^2\left(M_{n \times m}\left(F\right)\right) \slash \bigcap \ker f,
	\end{displaymath}
	where the intersection is over all elements $f \in \Hom_{\GL_n\left(F\right) \times \left\{1\right\}}\left(L^2\left(M_{n \times m}\left(F\right)\right), \tau\right)$.
	The group $\GL_m\left(F\right)$ acts on this space via $h \mapsto \rho\left(I_n, h\right)$. In the local field setting, the space $\Theta\left(\tau\right)$ is of finite length and admits a unique irreducible quotient called the small theta space, denoted $\theta\left(\tau\right)$. The assignment $\tau \mapsto \theta\left(\tau\right)$ has nice properties. In particular, it satisfies Howe reciprocity: the assignment $\tau \mapsto \theta\left(\tau\right)$ is injective if we restrict to $\tau$ such that $\theta\left(\tau\right)$ is not zero. For the type II case we concern, we also have that $\theta\left(\tau\right)$ is always non-zero if $n \le m$. One should think about this correspondence as an approximation for the decomposition of the space
	\begin{displaymath}
		L^2\left(M_{n \times m}\left(F\right)\right) \approx \bigoplus_{\tau \in \mathrm{Irr}\left(\GL_n\left(F\right)\right)} \tau \otimes \theta\left(\tau\right).
	\end{displaymath}
	We refer the reader to the introduction of \cite{CLLTZ} and the references within for more details.
	
	In \cite{GH2}, Gurevich--Howe studied the finite field analog of theta correspondence of type II. Suppose that $F = \bF$ is a finite field with $q$ elements, and let $L^2\left(M_{n \times m}\left(\bF\right)\right)$ be the space consisting of all functions $f \colon M_{n \times m}\left(F\right) \to \bC$. In this case, the big theta space can be identified with
	\begin{displaymath}
		\Theta\left(\tau\right) = \Hom_{\GL_n\left(F\right) \times \left\{1\right\}}\left(\tau \otimes 1, L^2\left(M_{n \times m}\left(\bF\right)\right)\right)
	\end{displaymath}
	with the $\GL_m\left(\bF\right)$-action on this space given by sending $h \in \GL_m\left(\bF\right)$ to left composition with $\rho\left(I_n, h\right)$. In this case, the big theta space $\Theta\left(\tau\right)$ does not have a unique irreducible quotient unless it is irreducible. Gurevich--Howe defined a substitute for $\theta\left(\tau\right)$, which they denoted $\eta\left(\tau\right)$, by pinpointing an interesting irreducible subrepresentation of $\Theta\left(\tau\right)$. In the case in hand, we may identify the $\GL_n\left(\bF\right) \times \GL_m\left(\bF\right)$ representation of $L^2\left(M_{n \times m}\left(\bF\right)\right)$ with $\bC$-linear combinations of elements of $M_{m \times n}\left(\bF\right)$ (up to composition with the outer automorphism) by sending a function $f$ to the linear combination $\sum_{X} f\left(X\right) \left[\transpose{X}\right]$.
	
	Thus Proposition~\ref{prop:SWdecomposition} can be viewed as an version of the theta correspondence of type II for monoid algebras (instead of group algebras). Interestingly, in the case in hand, the big theta space $\Theta\left(L_n\left(\pi\right)\right) = L_m\left(\pi\right)$ is always irreducible and we have an equality for the decomposition of the space instead of an approximation. This suggests that the theta correspondence for groups is a ``shadow'' of the theta correspondence for monoid algebras.
\end{remark}

Returning to the basic representation $V=k[\bF^n]$ of $k[\fM_n]$, we obtain the following result.

\begin{corollary}
Let $m \ge 0$. We have a decomposition of $k[\fM_n]$-modules
\begin{displaymath}
V^{\otimes m} = \bigoplus_{r = 0}^{\min\left(n,m\right)} \bigoplus_{\pi} L_n(\pi)^{\oplus \bm(\pi)}
\end{displaymath}
where $\pi$ runs over irreducible $\fG_r$-representations and
\begin{displaymath}
\bm(\pi) = \dim L_m(\pi) = \qbinom{m}{r} \cdot \dim{\pi}.
\end{displaymath}
\end{corollary}

\subsection{An $\bF$-linear ``restriction'' problem}

The restriction problem in combinatorial representation theory aims to understand the restriction of an irreducible $\GL_n(\bC)$-representation to the subgroup of permutation matrices $\fS_n$; specifically, the goal is to find a positive combinatorial interpretation for the multiplicity of a Specht module $S^\lambda$ inside a Weyl module $W_\mu$. This problem is considered open: there exist certain formulae \cite{Littlewood, ScharfThibon} that allow one to compute the multiplicities of Specht modules in a given Weyl module, but no satisfactory combinatorial answer is generally available at present, see \cite{Steinberg3}.

Let $\fF_{n,m}$ denotes the set of functions from $[m]$ to $[n]$. We have commuting actions of $\fS_n$ and $\fS_m$ on this set via post- and pre-composition respectively. If we identify $k[\fF_{n,m}]$ with $(k^n)^{\otimes m}$ and apply ordinary Schur--Weyl duality, we can decompose this as a $\GL_n(k) \times \fS_m$ representation. If we knew how to decompose $k[\fF_{n,m}]$ as an $(k[\fS_n], k[\fS_m])$-bimodule, then looking at an $\fS_m$-isotypic component would tell us how to restrict the corresponding Schur--Weyl dual representation from $\GL_n(k)$ to $\fS_n$. Hence we have the following equivalent formulation of the restriction problem:

\begin{problem}
Decompose $k[\fF_{n,m}]$ as a $(k[\fS_n], k[\fS_m])$-bimodule.
\end{problem}

From this perspective a natural $q$-analog of the restriction problem would be to decompose $k[\fM_{n,m}]$ as a $(\fG_n, \fG_m)$-bimodule. Normally one expects problems like this to be harder for finite general linear groups than for symmetric groups,  but amazingly in this case we can completely solve the $\fG_n$-analog of the restriction problem:

\begin{enumerate}
\item First decompose $k[\fM_{n,m}]$ as a $(k[\fM_n], k[\fM_m])$-bimodule according to Proposition~\ref{prop:SWdecomposition}.
\item Then restrict from $k[\fM_n] \times k[\fM_m]$ to $\fG_n \times \fG_m$ using the branching rule in Corollary~\ref{cor:RestrictionCombinatorics}.
\end{enumerate}

We note that in contrast to $\fM_n$, the monoid $\fF_n$ of all functions from $[n]$ to itself has non-semisimple representation theory.  This makes a monoid representation approach to the restriction problem more difficult, under this dictionary the restriction problem becomes more or less equivalent to computing the Cartan matrix for $k[\fF_n]$, which currently seems out of reach.

\appendix

\section{Discovery of the unit formula}

In this appendix, we explain how we used computational tools to discover the formula for the unit of the ideal $\fa_{n,r}$ of $k[\fM_n]$ (\S \ref{s:unit}).

In \cite{Kovacs}, Kov\'acs proved that such unit exists. To show this, he showed that such unit would be supported on semi-idempotent elements, and that a linear combination of such elements would be a unit if and only if the coefficients satisfy a linear system of equations. He showed that after reordering the semi-idempotent matrices by their rank sequences (in the lexicographic order), this linear system of equations is represented by an upper triangular matrix. He also proved that the diagonal entries of this matrix are powers of $q$. Therefore over a field of characteristic not dividing $q$, this system has a solution. See also \cite[Theorem 5.31]{Steinberg}.

In order to find the unit formula, we wrote a program in SageMath that explicitly solves Kov\'acs' linear system. This involves counting matrices obtained by modifying the last columns of a representative of a conjugacy class of a semi-idempotent element. We explain the details below.

If $A \in \fM_n$ is a semi-idempotent element, we define its rank sequence to be the sequence $R\left(A\right) = \left(\rank A^{j}\right)_{j=0}^{\infty}$. This is a weakly decreasing sequence of non-negative integers not greater than $n$. This sequence stabilizes with the value $\srk A$ before or at index $n$. Conjugacy classes of semi-idempotent elements are determined by their rank sequence.

The idea of finding a unit of $\fa_{n,r}$ is to use Lemma \ref{lem:kuhn-criterion-for-unit}. If $\eta_r$ is a unit of $\fa_{n,r}$ and we write
\begin{displaymath}
	\eta_r = \sum_{\substack{u \in \fM_n \\
			u \text{ semi-idempotent}\\
			\rank u \le r}} c_{u} \left[u\right],
\end{displaymath}
then Lemma \ref{lem:kuhn-criterion-for-unit} tells us that $c_{u} = c_{u'}$ if $u$ and $u'$ are conjugate and therefore we can write
\begin{displaymath}
	\eta_r = \sum_{\substack{\Gamma = \left(t_j\right)_{j=0}^{\infty}\\
			t_1 \le r}} \sum_{\substack{u \in \fM_n \\
			u \text{ semi-idempotent}\\
			R\left(u\right) = \Gamma}} c_{\Gamma} \left[u\right].
\end{displaymath}
To find the coefficients $c_{\Gamma}$, we use the second part of Lemma \ref{lem:kuhn-criterion-for-unit} which tells us that \begin{equation}\label{eq:mu-r-times-er-equals-er}
	\eta_r \left[\begin{pmatrix}
		I_{r}\\
		& 0_{n-r}
	\end{pmatrix}\right] = \left[\begin{pmatrix}
		I_{r}\\
		& 0_{n-r}
	\end{pmatrix}\right].
\end{equation} For every rank sequence $\Gamma$ we fix a representative $u_{\Gamma}$ with that rank sequence such that first $\rank u_{\Gamma}$ columns of $u_{\Gamma}$ are non-zero, and the rest of the columns are zero. Comparing the coefficient of $u_{\Gamma}$ of both sides of \eqref{eq:mu-r-times-er-equals-er}, we obtain the system of equations
\begin{displaymath}
	\sum_{\Gamma'} a_{\Gamma,\Gamma'} c_{\Gamma'} = \begin{cases}
		1 & \Gamma = \left(n,r,r,r,\dots\right)\\
		0 & \text{otherwise}
	\end{cases},
\end{displaymath}
where
\begin{displaymath}
	a_{\Gamma,\Gamma'} = \#\{ u \in \fM_n \mid u \text{ semi-idempotent}, R\left(u\right) = {\Gamma}', u \begin{pmatrix}
		I_{r}\\
		& 0_{n-r}
	\end{pmatrix} = u_{\Gamma}  \}.
\end{displaymath}

We set up this system in SageMath. The full code is available at \url{https://github.com/darkl/kovacs-unit}. We describe our core functions:
\begin{itemize}
	\item The function \texttt{is\_semi\_idempotent(T, polynomial\_ring, x)} checks whether the matrix \texttt{T} is semi-idempotent. It does so by checking whether the minimal polynomial of \texttt{T} is of the form $x^{\ell}\left(x-1\right)$ or $x^{\ell}$.
	\item The function \texttt{R(T)} returns the rank sequence of the matrix \texttt{T}.
	\item The function \texttt{partial\_jordan(field, partition)} receives a field and a partition $\lambda$ corresponding to the parameter \texttt{partition}. It returns a matrix $J$ over the field, such that $J$ is conjugate to $J_{\lambda}\left(0\right)$ (see \S \ref{ss:counting-semi-idempotents}), with the property that the first $\rank J$ columns of $J$ are non-zero and the rest of the columns are zero.
	\item The function \texttt{partial\_injective\_jordan(stable\_rank, nilpotent\_partition, n, q)} generates a representative of a semi-idempotent element $u_{\Gamma}$ with a specified stable rank and a given Jordan form corresponding to the zeroth eigenvalue, such that the first $\rank u_{\Gamma}$ columns of $u_{\Gamma}$ are non-zero and the rest are zero. This is achieved by creating a block diagonal matrix consisting of an identity matrix of size \texttt{stable\_rank} as the first block and of \texttt{partial\_jordan(GF(q), nilpotent\_partition)} as the second block.
	\item The function \texttt{semi\_idempotent\_types(n)} generates all tuples of $\left(i, \lambda\right)$ such that $i + \vert \lambda \vert = n$. These tuples are in bijection with all semi-idempotent conjugacy classes in $\fM_{n}$.
\end{itemize}

Given these functions at hand, we implement the function \texttt{get\_semi\_idempotent\_data(n, corank, q\_to\_sample)}. This function receives as its first argument the size of the matrices $n$, as its second input the corank ($n- r$) of interest and as its third input a list of $q$ to sample. It computes the coefficients $a_{\Gamma, \Gamma'}$ for the given collection of $q$ in the following manner: we call \texttt{semi\_idempotent\_types(n)} to get a collection representing the different conjugacy classes of semi-impotent elements in $\fM_{n}$. For each such conjugacy class we generate a nice matrix representative with coefficients in $\bF_q$ using the function \texttt{partial\_injective\_jordan}. We iterate over all the possible modifications the last $n-r$ columns of this representative and check if each modification is a semi-idempotent element of rank $\le n - r$. If it is, we compute the rank sequence of the modified element. Recording these statistics, we eventually compute $a_{\Gamma, \Gamma'}$. The output of the function \texttt{get\_semi\_idempotent\_data} is a nested dictionary with the structure $\Gamma \mapsto \Gamma' \mapsto q \mapsto a_{\Gamma, \Gamma'}\left(q\right)$.

Since we expect to see polynomial expressions in $q$, we interpolate our results. We define a function \texttt{get\_system\_of\_equations\_matrix(n, corank, q\_to\_sample)} that creates a matrix consisting of polynomials in the variable $q$, which we will denote $A_{n,r}\left(q\right)$, such that for the $q$ we sampled $A_{n,r}\left(q\right) = \left(a_{\Gamma, \Gamma'}\left(q\right)\right)$. This is done in two steps: first, the function \texttt{interpolate\_dict(dict\_input)} receives a nested dictionary that has the structure of the output of \texttt{get\_semi\_idempotent\_data} and returns a dictionary with the structure $\Gamma \mapsto \Gamma' \mapsto p_{\Gamma, \Gamma'}\left(q\right)$, where $p_{\Gamma, \Gamma'}\left(q\right)$ is a polynomial interpolating the values of $a_{\Gamma, \Gamma'}\left(q\right)$ for the sampled $q$. Then the function \texttt{dict\_to\_matrix} converts this dictionary to a square matrix $A_{n, r}\left(q\right)$.

The matrix $A_{n,r}\left(q\right)$ is such that the index corresponding to the semi-idempotent conjugacy class of $\left(\begin{smallmatrix}
	I_{r}\\
	& 0_{n-r}
\end{smallmatrix}\right)$ (which is parameterized by $\left(r, \left(1,1,\dots,1\right)\right)$) appears last. Thus the solution to our linear system of interest is the last column of $A_{n,r}\left(q\right)^{-1}$.

We first ran this code for small $n$ ($n=1,2,3,4,5$) and small $q$ ($q=2,3,4,5,7$) and corank $1$ ($r = n-1$):
\begin{lstlisting}
for n in range(1,6):
  print(f"n={n}:")
  (keys, system_matrix) = get_system_of_equations_matrix(n, 1, [2,3,4,5,7])
  matrix_inverse = system_matrix^(-1)
  last_column_interpolated = matrix_inverse.column(system_matrix.dimensions()[0] - 1)
  current_index = 0
  for key in keys:
    print(f"\t{key}:{last_column_interpolated[current_index]}")
    current_index = current_index + 1
  # n = 1:
  #   (1, 0): 1
  # n = 2:
  #   (2, 0, 0): -1/q
  #   (2, 1, 0): -1/q
  #   (2, 1, 1): 1/q
  # n = 3:
  #   (3, 0, 0, 0): 1/q^3
  #   (3, 1, 0, 0): 1/q^3
  #   (3, 1, 1, 1): -1/q^3
  #   (3, 2, 1, 0): 1/q^3
  #   (3, 2, 1, 1): -1/q^3
  #   (3, 2, 2, 2): 1/q^2
  # n = 4:
  #   (4, 0, 0, 0, 0): -1/q^6
  #   (4, 1, 0, 0, 0): -1/q^6
  #   (4, 1, 1, 1, 1): 1/q^6
  #   (4, 2, 0, 0, 0): -1/q^6
  #   (4, 2, 1, 0, 0): -1/q^6
  #   (4, 2, 1, 1, 1): 1/q^6
  #   (4, 2, 2, 2, 2): -1/q^5
  #   (4, 3, 2, 1, 0): -1/q^6
  #   (4, 3, 2, 1, 1): 1/q^6
  #   (4, 3, 2, 2, 2): -1/q^5
  #   (4, 3, 3, 3, 3): 1/q^3
  # ...
\end{lstlisting}
We observed that for $n=1,2,3,4,5$, the last column of the inverse of the matrix in interest consists of elements of the form $\pm q^{-j}$. Running computations for larger $n$ for enough values of $q$ takes a long time. Instead we can run the computations for $q=2$, invert the matrix, take its last column and extract from it the corresponding signs and exponents. This variant of the code runs quite fast. We gave Claude.ai the lists of values for $n=1,2,3,4$ and asked it to predict the values for $n=5$ and $n=6$ (\url{https://claude.ai/share/8a71810d-85bb-4514-b561-74f603843f10}). Its predictions were very close and nearly perfect. Although Claude.ai was not able to explain the pattern, it pointed out that the exponent is a difference of two binomial coefficients, the first one corresponding to $\binom{n}{2}$. Given that insight, we were able to fully identify the pattern for the case $r=n-1$.

Next, we ran a modified version of the code where our goal was to compute $\eta_{r}$ modulo $\fa_{n,r-1}$. This is achieved by modifying the function \texttt{get\_semi\_idempotent\_data(n, corank, q\_to\_sample)} so that it only considers the semi-idempotent matrices of rank exactly $n-r$ obtained by modifying the last $n-r$ columns of the given representative. In this case, the coefficients of the semi-idempotent elements were again negative powers of $q$, up to a sign. Feeding this data to Claude.ai was not helpful this time\\
(\url{https://claude.ai/share/55cd8c75-7572-4381-a06c-21b51a0034f4}).\\
It generated JavaScript code with a conjectural formula for special cases and tests. However, most of these tests failed. Instead, restricting to rank sequences of special types and looking up the sequences of exponents in OEIS allowed us to fully identify the formula for $\eta_{r}$ modulo $\fa_{n,r-1}$.

Next, running the code for the computation of $\eta_{r}$ for small $q$, small $n$ and small $n-r$, we observed that the unknown coefficients are $q$-binomial coefficients. Using this observation, we were eventually able to fully identify the formula for $\eta_{r}$.

\end{document}